\documentclass[draft]{amsart}

\usepackage{amsmath,amssymb}
\usepackage{hyperref} 

\newtheorem{theorem}{Theorem}[section]
\newtheorem*{theorem*}{Theorem}
\newtheorem{lemma}[theorem]{Lemma}
\newtheorem{proposition}[theorem]{Proposition}

\theoremstyle{definition}
\newtheorem{definition}[theorem]{Definition}

\theoremstyle{remark}
\newtheorem{remark}[theorem]{Remark}

\numberwithin{equation}{section}
\allowdisplaybreaks

\newcommand{\IC}{\ensuremath{\mathbb{C}}}
\newcommand{\IQ}{\ensuremath{\mathbb{Q}}}
\newcommand{\IR}{\ensuremath{\mathbb{R}}}
\newcommand{\IZ}{\ensuremath{\mathbb{Z}}}
\newcommand{\IN}{\ensuremath{\mathbb{N}}}
\newcommand{\IH}{\ensuremath{\mathbb{H}}}
\newcommand{\IP}{\ensuremath{\mathbb{P}}}
\newcommand{\ICprime}{\ensuremath{\mathbb{C}^\prime}}

\newcommand{\dd}{\ensuremath{\mathrm{d}}}
\newcommand{\id}{\ensuremath{\mathbf{1}}}
\newcommand{\minusid}{\ensuremath{(\mathbf{-1})}}
\newcommand{\EE}{\ensuremath{\operatorname{E}}}

\newcommand{\re}[1]{\ensuremath{{\operatorname{Re}\left(#1\right)}}}
\newcommand{\im}[1]{\ensuremath{{\operatorname{Im}\left(#1\right)}}}
\newcommand{\abs}[1]{\ensuremath{{\left\lvert#1\right\rvert}}} 
\renewcommand{\arg}[1]{\ensuremath{{\operatorname{arg}\left(#1\right)}}}
\newcommand{\SL}[1]{\ensuremath{{\mathrm{SL}\!\left(2, #1 \right)}}}
\newcommand{\Matrix}[4]{{\begin{pmatrix} #1 & #2 \\ #3 & #4 \end{pmatrix}}}
\newcommand{\OO}[1]{\ensuremath{\mathcal{O}\left( #1 \right)}}

\title{Eichler integrals for Maass cusp forms of half-integral weight}

\author{T.~M\"uhlenbruch}
\address{Department of Mathematics and Computer Science, \href{http://www.fernuni-hagen.de}{FernUniversit\"at in Hagen}, 58084 Hagen, Germany}
\email{\href{mailto:tobias.muehlenbruch@fernuni-hagen.de}{tobias.muehlenbruch@fernuni-hagen.de}}
\author{W.~Raji}
\address{\href{http://www.aub.edu.lb/fas/math/}{Department of Mathematics}, \href{http://www.aub.edu.lb/}{American University of Beirut}, Beirut and fellow at \href{http://www.cams.aub.edu.lb/}{Center for Advanced Mathematical Sciences}, Beirut, Lebanon.}
\email{\href{mailto:wr07@aub.edu.lb}{wr07@aub.edu.lb}}


\begin{document}

\begin{abstract}
In this paper, we define and discuss Eichler integrals for Maass
cusp forms of half-integral weight on the full modular group. We
discuss nearly periodic functions associated to the Eichler
integrals, introduce period functions for such Maass cusp forms, and
show that the nearly periodic functions and the period functions are
closely related. Those functions are extensions of the periodic
functions and period functions for Maass cusp forms of weight $0$ on
the full modular group introduced by Lewis and Zagier.
\end{abstract}

\maketitle


\section{Introduction}
\label{A}
Recall that modular cusp forms of weight $k \in 2\IN$ (for the group $\SL{\IZ}$) are holomorphic functions $u_\text{h}$ from the upper half-plane $\IH=\{z=x+iy;\; x,y \in \IR, \, y >0\}$ to $\IC$, satisfying the $u_\text{h}(z+1)=u_\text{h}(z)$ and $u_\text{h}(-1/z) = z^k \, u_\text{h}(z)$, and vanish as $y \to \infty$.
More details can be found in e.g.\ \cite{La76} and \cite{Za98}.

In the context of the Eichler-Shimura theorem, we attach to each modular cusp form a polynomial $p$ of degree $\leq k-2$.
One way to define it is by the following integral transformation:
\begin{equation}
\label{C2.1}
p(\zeta) := \int_0^{i\infty} (\zeta-z)^{k-2}
u_\text{h}(z) \, \dd z \qquad (\zeta \in \IC).
\end{equation}
This integral transformation goes back to Eichler in \cite{Ei57}.
One important property is that each period polynomial satisfies the identities
\begin{equation}
\label{C2.2}
\begin{split}
&p(\zeta) + \zeta^{k-2}p\left(\frac{-1}{\zeta}\right) =0
\quad \text{and} \\
&p(\zeta) + (\zeta+1)^{k-2} p\left(\frac{-1}{\zeta+1}\right) + \zeta^{k-2} p\left(\frac{-\zeta-1}{\zeta}\right)
= 0
\end{split}
\end{equation}
for each $\zeta \in \IC$.
Some more details on the Eichler-Shimura theorem can be found in e.g.\ \cite[Chapter V and VI]{La76}, \cite[\S1.1]{KZ84} and \cite{Za91}.

There exists another way to define the period polynomial, involving a variant of the above integral.
Consider the integral transformation
\begin{equation}
\label{C3.1}
f_\text{h}(\zeta) := \int_\zeta^{i\infty} (\zeta-z)^{k-2} u_\text{h}(z) \, \dd z
\qquad (\zeta \in \IH),
\end{equation}
defined only on the upper half-plane $\IH$.
$f_\text{h}$ is obviously a holomorphic function; and it easily seen that $f_\text{h}$ is periodic.
The period polynomial $p$ now appears in the calculation
\begin{align*}
f_\text{h}(\zeta) - \zeta^{k-2} \, f_\text{h}\left(\frac{-1}{\zeta}\right)
&=
\int_\zeta^{i\infty} (\zeta-z)^{k-2} u_\text{h}(z) \, \dd z  - \int_{\zeta}^{0} (\zeta - z)^{k-2} u_\text{h}(z)  \, \dd z \\
&=
\int_0^{i\infty} (\zeta-z)^{k-2} u_\text{h}(z) \, \dd z  =  p(\zeta)
\qquad (\zeta \in \IH).
\end{align*}

\medskip

One important extension of the Eichler-Shimura isomorphism was done by Lewis and
Zagier in \cite{LZ01}.
They found a one-to-one correspondence between Maass cusp forms of weight $0$
(for $\SL{\IZ}$) and functions called period functions.

Part of their main result is the following
\begin{theorem*}[{\cite{LZ01}}]
Let $s$ be a complex number with $\re{s} = \frac{1}{2}$. There is an
isomorphism between the following three function spaces:
\begin{enumerate}
\item
The space of Maass cusp forms of weight $0$ with eigenvalue $s(1-s)$ for $\SL{\IZ}$:
\emph{Maass cusp forms} of \emph{weight} $0$ for $\SL{\IZ}$ are real-analytic functions $u:\IH \to \IC$,
satisfying the transformation properties $u(z+1)=u(z)$ and $u(-1/z) = u(z)$, are eigenfunctions of the hyperbolic
Laplacian $\Delta_0 =-y^2 \left(\partial_x^2 + \partial_y^2\right)$ with eigenvalue $s(1-s)$, and vanish
as $y \to \infty$.
\item
The space of holomorphic functions $f$ on $\IC\smallsetminus \IR$, satisfying $f(z+1)=f(z)$ and bounded
by $\abs{\im{z}}^{-A}$ for some $A > 0$, such that the function $f(z) - z^{-2s} f(-1/z)$ extends holomorphically across
the positive real axis and is bounded by a multiple of $\min \big\{1, \abs{z}^{-1} \big\}$ in the right half-plane.
\item
The space of holomorphic solutions $\psi:\ICprime \to \IC$ of the three-term functional equation
\[
\psi(\zeta) = \psi(\zeta + 1) + (\zeta + 1)^{2s} \psi\left( \frac{\zeta}{\zeta+1} \right)
\]
on $\ICprime = \IC \smallsetminus \IR_{\leq0}$ which satisfy the growth condition
\[
\psi(\zeta) =
\begin{cases}
\OO{\frac{1}{\abs{\zeta}}}  & \text{as $\zeta \to \infty$, $\re{\zeta >0}$ and} \\
\OO{1}  & \text{as $\zeta \to 0$, $\re{\zeta >0}$.}
\end{cases}
\]
\end{enumerate}
\end{theorem*}
In analogy to the case of modular cusp forms, Lewis and Zagier
also introduce two integral transformations in \cite[Chapter~II]{LZ01} from Maass cusp forms to period functions
(which are the functions $\psi$ above) and the periodic holomorphic functions $f$ on $\IC \smallsetminus \IR$.

\medskip

In this paper, we extend the integral transformations of \cite{LZ01}
to the context of Maass forms of half-integral weights with a
multiplier system. Our main result is the following:
\begin{theorem}
\label{A.1} Let $\nu$ be a purely imaginary complex number, i.e.\
$\nu \in i\IR$, $k \in \frac{1}{2} \IZ$ a weight and $v$ a
compatible multiplier as defined in \S\ref{B2}.

For each Maass cusp form $u$ of weight $k$, eigenvalue $\frac{1}{4}-\nu^2$ and multiplier $v$ for $\SL{\IZ}$,
see Definition~\ref{D1.1} for details, there exist the following functions:
\begin{enumerate}
\item
A holomorphic function function $f:\IC\smallsetminus \IR \to \IC$, which is nearly periodic, i.e.\ $f(z+1) = a f(z)$
for some $a\in \IC$ with $\abs{a}=1$, such that $f(z) - v\left(\Matrix{0}{-1}{1}{0} \right) \, z^{2\nu-1} \, f(-1/z)$
extends holomorphically across the positive real axis and is bounded by a multiple of $\min \{1, \abs{z}^{-1} \}$ in
the right half-plane.
\item
A holomorphic solution $P:\ICprime \to \IC$ of the three-term functional equation
\[
P(\zeta) = v\left( \Matrix{1}{1}{0}{1} \right)^{-1} \, P(\zeta + 1) + v\left( \Matrix{1}{0}{1}{1} \right)^{-1} \,
(\zeta + 1)^{2\nu-1} \, P\left( \frac{\zeta}{\zeta+1} \right)
\]
on $\ICprime$ which satisfies the growth condition
\[
P(\zeta) =
\begin{cases}
\OO{\frac{1}{\abs{\zeta}}}  & \text{as $\zeta \to \infty$, $\re{\zeta >0}$ and} \\
\OO{1}  & \text{as $\zeta \to 0$, $\re{\zeta >0}$.}
\end{cases}
\]
We call such a function $P$ a period function.
\end{enumerate}
\end{theorem}

The proof of this theorem follows ideas presented above for modular cusp forms for positive even weight:
We use a Maass-Selberg differential form to define the kernel of integral transformations similarly as
$(\zeta-z)^{k-2} u_\text{h}(z) \, \dd z$ is used above.
We then describe properties of the determined nearly periodic functions and period functions, i.e.,
the images of our integral transformations.
Summarizing, we introduce the arrows of the following diagram and show that the diagram commutes:

{\centering
\setlength{\unitlength}{1.2cm}
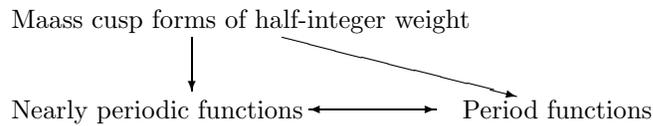
\begin{picture}(9, 2)
\put(1, 1.5){Maass cusp forms of half-integer weight}
\put(1, 0.5){Nearly periodic functions}
\put(6, 0.5){Period functions}

\put(3,1.4){\vector(0,-1){0.6}}
\put(4,1.4){\vector(4,-1){2.6}}
\put(4.7,0.6){\vector(1, 0){1}}
\put(5.3,0.6){\vector(-1, 0){1}}
\end{picture}
}

It is our hope that the results of this paper form a first step towards a working Eichler-Shimura theory for Maass cusp forms of half-integral weight (for $\SL{\IZ}$) since we establish one direction of a possible bijection between Maass cusp forms and period functions.
We will discuss this and some other related questions briefly in \S\ref{F}.

\medskip

The paper is organized as follows: Section~\ref{B} contains the
preliminaries, like defining properly the group $\SL{\IZ}$ and its
linear fractional transformations, the multiplier systems and the
slash and double-slash notations. In \S\ref{D1} we define the Maass
cusp forms for half-integral weight. The next section introduces the
$R$-function and the Maass-Selberg form. The sections~\ref{D3} and
\ref{D4} contain the definitions of the integral transformations
from Maass cusp forms to nearly periodic functions on one hand and
to period functions on the other hand. These sections contain also
our main result in a more detailed version.
In \S\ref{G} we collect those results and prove Theorem~\ref{A.1}.
We use \S\ref{E} to compare and relate our integral transforms to the one appearing in
the setting of the classical modular cusp forms. The remaining
Section~\ref{F} contains a short discussion and outlook.

\section{Preliminaries}
\label{B}
\subsection{The matrix group $\SL{\IZ}$ and its linear fractional transformations}
\label{B1} Let $\SL{\IR}$ denote the group of $2 \times 2$ matrices
with real entries and determinant~$1$. The subgroup $\SL{\IZ}\subset
\SL{\IR}$ denotes the \emph{full modular group}, that is the
subgroup of matrices with integer entries. It is generated by
\begin{equation}
\label{B.1}
S = \Matrix{0}{-1}{1}{0} \quad \text{and} \quad T = \Matrix{1}{1}{0}{1}.
\end{equation}
These satisfy
\begin{equation}
\label{B.5}
S^2 = (ST)^3 = \minusid,
\end{equation}
where $\minusid \in \SL{\IZ}$ is the matrix with $-1$ on the diagonal and $0$ on the off-diagonal entries.
$\id$ denotes the identity matrix.
We denote
\begin{equation}
\label{B.2}
T^\prime := TST = \Matrix{1}{0}{1}{1}.
\end{equation}
Note that
\begin{equation}
\label{B.6}
ST = \Matrix{0}{-1}{1}{1}
\quad \text{and} \quad
STST = T^{-1}S = \Matrix{-1}{-1}{1}{0}.
\end{equation}

The group $\SL{\IR}$ acts on the upper half-plane $\IH = \{ z \in \IC; \; \im{z} >0\}$ and its boundary $\IP_\IR = \IR \cup \{\infty\}$ and the lower half-plane $\IH^- = \{z \in \IC;\; \im{z} < 0\}$ by \emph{fractional linear transformations}
\begin{equation}
\label{B.3} \Matrix{a}{b}{c}{d} \, z  :=
\begin{cases}
\frac{a}{c} & \text{if $z = \infty$,} \\
\infty & \text{if $z = -\frac{d}{c}$ with $c \neq 0$, and} \\
\frac{az+b}{cz+d} &\text{otherwise}.
\end{cases}
\end{equation}
We also need the following $\mu$-function:
\begin{equation}
\label{B.13}
\mu:\SL{\IR} \times \IC \to \IC; \qquad \mu\left( \Matrix{a}{b}{c}{d},z\right) := cz+d.
\end{equation}
Obviously $\mu$ satisfies the cocycle-relation
\[
\mu(\gamma\delta,z) = \mu(\gamma,\delta \,z) \, \mu(\delta,z)
\]
for every $\gamma, \delta \in \SL{\IR}$.

Moreover, we have
\begin{equation}
\label{B.4}
\begin{split}
&\im{\gamma  z} = \frac{\im{z}}{\abs{\mu(\gamma,z)}^2}, \quad
\frac{\dd}{\dd z} \gamma z = \frac{1}{(\mu(\gamma,z))^2}, \\
&\frac{\dd}{\dd \bar{z}} \gamma \bar{z} = \frac{1}{(\mu(\gamma,\bar{z}))^2}
\quad \text{and} \quad
\gamma \zeta - \gamma z = \frac{\zeta - z}{\mu(\gamma,\zeta) \, \mu(\gamma,z)}
\end{split}
\end{equation}
for every $\gamma \in \SL{\IR}$.

\subsection{Multiplier systems}
\label{B2}
We call a function $ v:\SL{\IZ} \to \IC_{\neq 0}$ \emph{multiplier} or \emph{multiplier system} compatible with the half-integral \emph{weight} $k$ if $v$ satisfies
\begin{equation}
\label{B.10}
v(\gamma \delta) \, e^{ik \arg{\mu(\gamma\delta,z)}}
=
v(\gamma)v(\delta) \,e^{ik \arg{\mu(\gamma,\delta\,z)}} e^{ik \arg{\mu(\delta,z)}}
\end{equation}
for every $\gamma, \delta \in \SL{\IZ}$ and $z \in \IH$.

\begin{remark}
\label{B.9}
\begin{enumerate}
\item
The range of $\arg{\cdot}$ is $-\pi < \arg{z} \leq \pi$ for all $z \in \IC_{\neq0}$.
\item
Condition~\eqref{B.10} implies that the system of equations
\begin{equation}
\label{B.7} f(\gamma z) = v(\gamma) \,e^{ik \arg{\mu(\gamma,z)}}\, f(z) \qquad (z \in \IH, \gamma \in \SL{\IZ}),
\end{equation}
allows non-zero solutions $f:\IH \to \IC$.
\item
We have in particular
\begin{equation}
\label{B.11} v\big(\minusid\big) = e^{-ik\pi},
\end{equation}
since \eqref{B.7} with $\gamma = \minusid$ implies $v\big(\minusid\big) \,e^{ik \arg{-1}} = 1$, if $f$ does not vanish everywhere.
\end{enumerate}
\end{remark}

\subsection{The slash and double-slash notations}
\label{B3}
We define arbitrary powers $z^s$ with (possibly complex) $s$ by using the standard branch
of the logarithm:
$z^s =\abs{z}^s \, e^{is \arg{z}}$ with $\arg{z} \in (-\pi,\pi]$ for every
$z \in \IC_{\neq 0}$.

We introduce the \emph{slash} and the \emph{double-slash} notations.
Let $k \in \frac{1}{2}\IZ$, $\nu \in \IC$ and $v$ be a multiplier.
For $f:\IH \to \IC$ and $\gamma \in \SL{\IZ}$ we define
\begin{equation}
\label{B.12}
\begin{split}
\left(f\big|_k^v \gamma \right) (z)
&:=
e^{-ik \arg{\mu(\gamma,z)}} \, v(\gamma)^{-1} \, f(\gamma \,z)  \qquad \text{and} \\
\left(f \big\|_{\nu}^v \gamma \right)(z)
&:=
v(\gamma)^{-1}  \big(\mu(\gamma, z)\big)^{2\nu-1} \, f(\gamma \, z)
\end{split}
\end{equation}
for every $z \in \IH$.
For example \eqref{B.7} reads as $f\big|_k^{v} \gamma  = f$.

We define the slash and double-slash notations also for functions
$f:\IH^- \to \IC$ on the lower half-plane, since $\gamma\, z \in
\IH^-$ in \eqref{B.3} for every $\gamma \in \SL{\IZ}$ and $z \in
\IH^-$.

As slight abuse of notation, we also use the slash notation
\begin{equation}
\label{B.15}
\left(f\big|_k^1 \gamma \right) (z) = e^{-ik \arg{\mu(\gamma,z)}} \, f(\gamma \,z)
\end{equation}
for matrices $\gamma \in \SL{\IR}$.

Consider the subset $\SL{\IZ}^+ \subset \SL{\IZ}$, containing all
matrices $\gamma \in \SL{\IZ}$ with only nonnegative entries, i.e.,
all $\Matrix{a}{b}{c}{d} \in \SL{\IZ}$ satisfying $a,b,c,d \geq 0$.
These matrices have the property that they map the cut-plane
$\ICprime = \IC \smallsetminus (-\infty,0]$ into itself: for every
$z \in \ICprime$ and $\gamma \in \SL{\IZ}^+$ we have $\gamma z \in
\ICprime$. The slash and double-slash notations in \eqref{B.12} are
also well defined for functions $f:\ICprime \to \IH$ and all $\gamma
\in \SL{\IZ}^+$. For given real $z$ we may even extend the slash and
double-slash notations to certain matrices $\gamma \in \SL{\IZ}$
which satisfy $\mu(\gamma,z)>0$ on occasion.

\begin{remark}
\label{B.14}
The slash notation in \eqref{B.12} can be viewed as a group operation of $\SL{\IZ}$
on the space of functions on the upper half-plane.
Indeed, relation~\eqref{B.10} implies that
$f\big|_k^v \big(\gamma\delta\big)
= \left(f\big|_k^v \gamma\right)\big|_k^v\delta$ holds.

The double-slash notation in \eqref{B.12} is, as the name indicates,
an abbreviation for the given expression.
It is not a group operation in general.
The same is true for the slash-notation except in the case mentioned above.
\end{remark}


\section{Maass Cusp Forms of Half-Integral Weight and Maass Operators}
\label{D1}
In this section, we define Maass cusp form and useful Maass operators that will be used later.
As usual, we write $z = x + iy$ for complex $z$ with real part $x$ and imaginary part $y$.

\begin{definition}
\label{D1.1}
Let $k$ be a half-integral weight and $v: \SL{\IZ} \to \IC_{\neq
0}$ a compatible multiplier system. A \emph{Maass cusp
form} of weight $k$ and multiplier $v$ for $\SL{\IZ}$ is a
real-analytic function $u:\IH \to \IC$ satisfying
\begin{enumerate}
\item \label{D1.1.1}
$u\big|_k^v\gamma =u$ for every $\gamma \in \SL{\IZ}$,
\item \label{D1.1.2}
$u$ is an eigenfunction of the Laplace operator $\Delta_k$ with
eigenvalue $\lambda$, i.e., $\Delta_k u = \lambda \, u$ where
\begin{equation}
\label{D1.2} \Delta_k = -y^2 \left(\partial^2_x + \partial^2_y
\right) + ik y \partial_x
\end{equation}
with $z= x+iy \in \IH$.
\item \label{D1.1.3}
$u$ satisfies the growth condition $u(z) = \OO{y^c}$ as $y \to
\infty$ for every $c \in \IR$.
\end{enumerate}
\end{definition}

It is known, see for example \cite{Br94}, that the eigenvalue
$\lambda$ is real.
It is convenient to write $\lambda = \frac{1}{4} - \nu^2$
with suitable \emph{spectral parameter} $\nu \in \IR \cup
i\IR$.

In the following lemma, we extend the definition of the Maass cusp
form to the lower half-plane by considering the conjugate of the
form defined on the upper half-plane.
\begin{lemma}
\label{D1.8} Let $u$ be a Maass cusp form of weight $k$, multiplier
system $v$ and eigenvalue $\lambda$.
Defining $\tilde{u}: \IH^-\to
\IC$; $z \mapsto \tilde{u}(z):= u(\bar{z})$ for a Maass cusp form
$u$ defines a real-analytic function on the lower half-plane which
satisfies
\begin{enumerate}
\item \label{D1.8.1}
$\tilde{u}\big|_{-k}^v\gamma =\tilde{u}$ for every $\gamma \in
\SL{\IZ}$,
\item \label{D1.8.2}
$\tilde{u}$ is an eigenfunction of the Laplace operator
$\Delta_{-k}$ with eigenvalue $\lambda$, and
\item \label{D1.8.3}
$\tilde{u}$ satisfies the growth condition $\tilde{u}(z) =
\OO{|y|^c}$ as $y \to -\infty$ for every $c \in \IR$.
\end{enumerate}
\end{lemma}

\begin{proof}
Using the identity $\arg{\bar{\zeta}} = - \arg{\zeta}$, $\zeta \in
\IC \smallsetminus (-\infty, 0]$, the transformation property
follows immediately:
\[
\tilde{u}\left( \gamma \, z \right)
=
u\left( \gamma \, \bar{z} \right)
=
e^{ik\arg{\mu(\gamma,\bar{z})}} \, v(\gamma) \, u(\bar{z})
=
e^{i(-k)\arg{\mu(\gamma,z)}} \, v(\gamma) \, \tilde{u}(z)
\]
for every $z \in \IH^-$ and $\gamma \in
\SL{\IZ}$.
The substitution $y \mapsto -y$ (i.e.\ $z \mapsto
\bar{z}$) gives
\begin{align*}
&
\Delta_{-k} \tilde{u} (z) \\
&=
\left[ -y^2 \left(\partial^2_x + \partial^2_y \right) + i(-k) y \partial_x \right] u(\bar{z}) \\
&=
\left[ -(-y)^2 \left(\partial^2_x + (-1)^2 \partial^2_y \right) + i(-k) (-y) \partial_x \right] u(z)
\qquad (\text{using } y \mapsto -y)\\
&= \Delta_k \, u(z) = \lambda \, u(z) = \lambda \, u(\bar{z})
\qquad (\text{using } y \mapsto -y)\\
&= \lambda \, \tilde{u} (z)
\end{align*}
for $z \in \IH^-$. This shows the second property. The growth
condition follows directly from the definition $\tilde{u}(z) =
u(\bar{z})$.
\end{proof}

\medskip

The raising and lowering \emph{Maass operators} acting on the space
of cusp forms of given weight $k$, multiplier $v$, and eigenvalue
$\lambda$ are given by
\begin{equation}
\label{D1.4} \EE^\pm_{k} = \pm 2iy\partial_x +2y\partial_y \pm k.
\end{equation}
Equivalently, it is sometimes convenient to write
\begin{equation}
\label{D1.5} \EE^+_k = 4iy\partial_z +k \quad \text{and} \quad
\EE^-_k = -4iy\partial_{\bar{z}} -k.
\end{equation}
They satisfy the identity
\begin{equation}
\label{D1.9}
\EE^\pm_{k \mp 2} \EE^\mp_k = -4\Delta_k - k(k \mp 2).
\end{equation}
Thus if $u:\IH \to \IC$ is an eigenfunction of $\Delta_k$ with
spectral value $\nu$, i.e.~$\Delta_k u = \left( \frac{1}{4}-\nu^2
\right) u$, then $u$ satisfies
\begin{equation}
\label{D1.3}
\EE^\pm_{k \mp 2} \EE^\mp_k  \, u = \big(1+2\nu \mp k \big)\big(-1+2\nu \pm k \big) u.
\end{equation}

\medskip

The slash notation defined in \eqref{B.12} commutes with the Laplace
operator,
\begin{equation}
\label{D1.6} \Delta_k\big( f\big|_{k}^v \gamma \big) =
\big(\Delta_kf\big)\big|_{k}^v \gamma,
\end{equation}
and interacts as follows with the Maass-operators
\begin{equation}
\label{D1.7} \EE^\pm_{k}\big( f\big|_{k}^v \gamma \big) =
\big(\EE^\pm_{k}f\big)\big|_{k\pm 2}^v \gamma \qquad(\gamma \in
\SL{\IZ}),
\end{equation}
for every $k\in\IR$ and $u$ real-analytic.

\section{The Maass-Selberg Differential Form and the $R$-function}
\label{D2} We need to define the Maass-Selberg differential form
that will be used later to define the kernel of the associated
integrals of the Maass cusp forms. First we define what is known as
the $R$-function. It will play an important role in the construction
of the kernel.

\subsection{The $R$-function}
\label{D2.32}
We define $h(z):=\im{z}$ for $z \in \IH$.
For $k\in \frac{1}{2}\IZ$ and $\nu \in \IC$, it is easy to see that $h(z)$ is
real-analytic and positive for $z\in \IH$, and that $h$ satisfies
the differential equations
\begin{equation}
\label{D2.3}
\Delta_k h^{\frac{1}{2}-\nu} = \bigg(
\frac{1}{4} - \nu^2 \bigg) \, h^{\frac{1}{2}-\nu} \quad \text{and} \quad
\EE_{k}^\pm h^{\frac{1}{2}-\nu} = ( 1-2\nu \pm k ) \, h^{\frac{1}{2}-\nu}.
\end{equation}

Define
\begin{equation}
\label{D2.5}
R_{k,\nu}(z,\zeta)
:=
\left(\frac{\sqrt{\zeta - z\,}}{\sqrt{\zeta - \bar{z}\,}}\right)^{-k}
    \left( \frac{\abs{\im{z}}}{(\zeta-z)(\zeta-\bar{z})} \right)^{\frac{1}{2}-\nu}.
\end{equation}
for $\zeta, z \in \IC$ such that
\begin{equation}
\label{D2.2}
\zeta-z, \, \zeta -\bar{z} \not\in \IR_{\leq 0}
\end{equation}
holds.

\begin{remark}
\label{D2.1}
The square roots $\sqrt{\zeta - z\,}$ and $\sqrt{\zeta - \bar{z}\,}$ on the right hand
side of \eqref{D2.5} are well defined since we require that $\zeta - z$ and $\zeta - \bar{z}$ are
in $\IC \smallsetminus \IR_{\leq 0}$.
The square roots in this situation, interpreted as principal square roots, are holomorphic.
\end{remark}

The $R$-function has the following properties:
\begin{proposition}
\label{D2.21}
\begin{enumerate}
\item
The function
\[
z \mapsto R_{k,\nu}(z,\zeta)
\]
is smooth in the real and imaginary part of $z$ if \eqref{D2.2} holds.
\item
The map
\[
\zeta \mapsto R_{k,\nu}(z,\zeta)
\]
is holomorphic on $\IC \smallsetminus \{z-r,\overline{z}-r;\; r \geq 0\}$.
\item
$R_{k,\nu}$ has the form
\begin{equation}
\label{D2.4}
\begin{split}
R_{k,\nu}(z,\zeta)
&=
e^{-ik\arg{\zeta-z}} \left( \frac{\im{z}}{(\zeta-z)(\zeta-\bar{z})} \right)^{\frac{1}{2}-\nu} \\
&=
\left( h^{\frac{1}{2}-\nu}\big|_k^1 \Matrix{0}{1}{-1}{\zeta}\right) (z)
\end{split}
\end{equation}
for real $\zeta$ and $z \in \IH$.
\item
Assume the usual restriction \eqref{D2.2} for $z$ and $\zeta$.
The function
\[
\IH \to \IC; \qquad z \mapsto R_{k,\nu}(z,\zeta)
\]
satisfies the differential equations
\begin{align}
\nonumber
\Delta_k R_{k,\nu}(\cdot,\zeta)
&=
\bigg( \frac{1}{4}-\nu^2 \bigg) \, R_{k,\nu}(\cdot,\zeta)    \quad \text{and}\\
\label{D2.6}
\EE^\pm_k  R_{k,\nu}(\cdot,\zeta)
&=
(1-2\nu \pm k) \,R_{k\pm 2,\nu}(\cdot,\zeta).
\end{align}

The function
\[
\IH^- \to \IC; \qquad z \mapsto R_{k,\nu}(z,\zeta)
\]
satisfies
\begin{align*}
\Delta_{-k} R_{k,\nu}(\cdot,\zeta)
&=
\bigg( \frac{1}{4}-\nu^2 \bigg) \, R_{k,\nu}(\cdot,\zeta)    \quad \text{and}\\
\EE^\pm_{-k}  R_{k,\nu}(\cdot,\zeta)
&=
(1-2\nu \pm k) \, R_{k\pm 2,\nu}(\cdot,\zeta).
\end{align*}
\end{enumerate}
\end{proposition}

\begin{proof}
\begin{enumerate}
\item
For fixed $\zeta \in \IC$, assume that $z \in \IC$ satisfies
\eqref{D2.2}. It is then obvious from \eqref{D2.5} that the function
$z \mapsto R_{k,\nu}(z,\zeta)$ is smooth in the real and the
imaginary parts of $z$. (Observe that the values under the
square-roots are never negative by condition \eqref{D2.2}.\@)
\item
Fix $z \in \IC$ this time. Again, it is obvious from \eqref{D2.5}
that the function $\zeta \mapsto R_{k,\nu}(z,\zeta)$ is holomorphic
for all $\zeta \in \IC$ satisfying condition \eqref{D2.2}. Noticing
that $\zeta$ satisfies \eqref{D2.2} is equivalent to the fact that
$\zeta$ is an element of the ``two-cut plane'' $\IC \smallsetminus
\big(z + \IR_{\leq0}  \cup \overline{z} +\IR_{\leq0}\big)$. This
shows the second part of the proposition.
\item
The first equality follows by rewriting the right hand side of \eqref{D2.5} using the identity
\begin{equation}
\label{D2.31}
\frac{\sqrt{\zeta-z}}{\sqrt{\zeta-\bar{z}}}
= \frac{\zeta-z}{\sqrt{\zeta-z}\sqrt{\zeta-\bar{z}}}
= e^{i \arg{\zeta-z}},
\end{equation}
which is correct under the given assumptions $\zeta \in \IR$ and $z \in \IH$.
The second equality follows by the slash notation in \eqref{B.15}.
\item
We assume real $\zeta$ and $z \in \IH$ for the moment.
Combining the last expression of $R_{k,\nu}(z,\zeta)$ in \eqref{D2.4} with \eqref{D1.6} and then with \eqref{D2.3}
gives
\begin{align*}
\Delta_k R_{k,\nu}(\cdot,\zeta)
&=
\Delta_k \left( h^{\frac{1}{2}-\nu}\big|_k^1 \Matrix{0}{1}{-1}{\zeta} \right) \\
&=
\left( \Delta_k h^{\frac{1}{2}-\nu} \right)\big|_k^1 \Matrix{0}{1}{-1}{\zeta}  \\
&=
\big( \frac{1}{4} - \nu^2 \big) \, h^{\frac{1}{2}-\nu}\big|_k^1 \Matrix{0}{1}{-1}{\zeta} \\
&=
\big( \frac{1}{4} - \nu^2 \big) \, R_{k,\nu}(\cdot,\zeta).
\end{align*}
Analogously, just using \eqref{D1.7} instead of \eqref{D1.6} shows
that
\[
\EE^\pm_k  R_{k,\nu}(\cdot,\zeta)
=
(1-2\nu \pm k) R_{k\pm 2,\nu}(\cdot,\zeta).
\]

Next, using the substitution arguments in the proof of Lemma~\ref{D1.8} shows that
$R_{k,\nu}(z,\zeta)$ satisfies the stated properties for $z \in \IH^-$.

The last step is to extend $\zeta$ from real values to complex
values. This can be done since $\zeta$ is just a constant for the
differential operators. This shows that the stated differential
equations hold also for complex $\zeta$ as long as $z$ and $\zeta$
satisfy the condition \eqref{D2.2}.
\end{enumerate}
\end{proof}

\begin{remark}
\label{D2.24}
The $R$-function appeared first in \cite{LZ01}, where Lewis and Zagier introduced
\[
R_{0,-\frac{1}{2}}(z,\zeta) = \frac{y}{(x-\zeta)^2+y^2} =
\frac{i}{2}\left( \frac{1}{z-\zeta} - \frac{1}{\overline{z}-\zeta}
\right)
\]
in  \cite[p.211, above (2.6)]{LZ01}.
Their notation for the above expression was $R_\zeta(z)$.
\end{remark}

Before we describe the transformation law of
$R_{k,\nu}(z,\zeta)$, we need one trivial auxiliary lemma which will allow us to perform a certain
factorization in the proof of the forthcoming Lemma~\ref{D2.26}.
\begin{lemma}
\label{D2.23}
Let $z,w \in \ICprime$ and $\alpha \in \IC$ be complex numbers satisfying either
\begin{enumerate}
\item $z\in \IR_{>0}$ or
\item the product $zw\in \IR_{>0}$.
\end{enumerate}
Then the identity $\big(z \,w \big)^\alpha = z^\alpha \, w^\alpha$ holds.
\end{lemma}

\begin{proof}
\begin{enumerate}
\item
Assume that $z$ is real and positive and $w \in \ICprime$.
This ensures $\arg{w} = \arg{zw} \in (-\pi,\pi)$.
Writing $w$ in its polar coordinates gives
\[
\big(z \,w \big)^\alpha
=
\bigg( \big(z \abs{w}\big) \, e^{i\arg{w}} \bigg)^\alpha
=
z^\alpha \, \bigg( \abs{w} \, e^{i\arg{w}} \bigg)^\alpha
=
z^\alpha \, w^\alpha.
\]
\item
Assume that $zw$ is real and positive and both $z,w \in \ICprime$.
This ensures $\arg{z} = -\arg{w} \in (-\pi,\pi)$.
Writing $z$ and $w$ in its polar coordinates gives
\begin{align*}
\big(z \,w \big)^\alpha \; z^{-\alpha}
&=
\big( \abs{zw} \big)^\alpha \; z^{-\alpha}
=
\bigg( \abs{z} \, e^{i\arg{z}} \; \abs{w} \, e^{i\arg{w}} \bigg)^\alpha \; \bigg( \abs{z}^{-1} \, e^{-i\arg{z}} \bigg)^\alpha \\
&=
\bigg( \abs{z} \, e^{i\arg{z}} \; \abs{w} \, e^{i\arg{w}} \bigg)^\alpha \; \bigg( \abs{z}^{-1} \, e^{-i\arg{z}} \bigg)^\alpha
\quad \text{(using case (1))} \\
&=
\bigg( \abs{z} \, e^{i\arg{z}} \; \abs{w} \, e^{i\arg{w}} \;  \abs{z}^{-1} \, e^{-i\arg{z}} \bigg)^\alpha \\
&=
\bigg( \abs{w} \, e^{i\arg{w}} \bigg)^\alpha
=
w^\alpha,
\end{align*}
where we used that $zw \in \IR_{>0}$, $z^{-1} \in \ICprime$ implies
\[
\big(zw\, z^{-1}\big)^\alpha = (zw)^\alpha \, \big(z^{-1}\big)^\alpha
\]
as shown above.
\end{enumerate}
The identity $(z\, w)^\alpha = z^\alpha \, w^\alpha$ holds in both situations.
\end{proof}

\begin{remark}
\label{D2.30}
It is important that $z,w \not\in \IR_{<0}$ in the second case of the above auxiliary lemma.
If $z,w$ are both real and negative, then $zw$ itself is positive.
Due to the choice involved in the argument function, see Remark~\ref{B.9}, the arguments $\arg{z}=\pi$ and $\arg{w}=\pi$ are not anymore of opposite sign: $\arg{z} \neq -\arg{w}$.
Hence the factorization in the proof of the auxiliary lemma does not work anymore.
\end{remark}

In what follows, we show the transformation law of
$R_{k,\nu}(\zeta,z)$.
\begin{lemma}
\label{D2.26} Let $\gamma \in \SL{\IZ}$, $\zeta, z \in \IC$
satisfying~\eqref{D2.2} and $\mu(\gamma,\zeta), \, \mu(\gamma,z) \in
\ICprime$ with $\re{\mu(\gamma,\zeta)} > 0$. Moreover, assume that
$\zeta$ and $z$ satisfy one of the following three conditions:
\begin{enumerate}
\item $\mu(\gamma,\zeta) \in \IR_{>0}$,
\item $\zeta \in \IH$ and $\gamma z \in \gamma \zeta  + i\IR_{> 0}$ or
\item $\zeta \in \IH^-$ and $\gamma \bar{z} \in \gamma \bar{\zeta}  + i\IR_{> 0}$.
\end{enumerate}
Then, the function
$(\zeta,z) \mapsto R_{k,\nu}(z,\zeta)$ satisfies the transformation
formula
\begin{equation}
\label{D2.27} R_{k,\nu}(\gamma z,\gamma \zeta)
=
e^{ik\arg{\mu(\gamma,z)}} \big(\mu(\gamma,\zeta)\big)^{1-2\nu} \, R_{k,\nu}(z,\zeta).
\end{equation}
\end{lemma}

\begin{proof}
Take $\gamma \in \SL{\IZ}$ and $\zeta, z \in \IC$ satisfying~\eqref{D2.2} and $\mu(\gamma,z) \in \ICprime$.

One key observation is the fact that the factorization
\begin{equation}
\label{D2.13}
\left( \frac{\abs{\im{\gamma z}}}{(\gamma \zeta- \gamma z)(\gamma \zeta- \gamma \bar{z})} \right)^{\frac{1}{2}-\nu}
=
\big(\mu(\gamma,\zeta)\big)^{1-2\nu}\, \left( \frac{\abs{\im{z}}}{(\zeta- z)(\zeta- \bar{z})} \right)^{\frac{1}{2}-\nu}
\end{equation}
holds if $\zeta$ and $z$ satisfy one of the additional assumptions.

First, we assume $\mu(\gamma,\zeta) \in \IR_{>0}$.
Using identities in \eqref{B.4} and Lemma~\ref{D2.23} gives
\begin{align*}
\left( \frac{\abs{\im{\gamma z}}}{(\gamma \zeta- \gamma z)(\gamma \zeta- \gamma \bar{z})} \right)^{\frac{1}{2}-\nu}
&=
\left( \frac{\frac{\abs{\im{z}}}{\mu(\gamma,z) \, \mu(\gamma,\bar{z})}}{\frac{\zeta- z}{\mu(\gamma,\zeta)\, \mu(\gamma,z)}
  \frac{\zeta- \bar{z}}{\mu(\gamma,\zeta) \, \mu(\gamma,\bar{z})}} \right)^{\frac{1}{2}-\nu} \\
&=
\left( \frac{\abs{\im{z}}}{\frac{\zeta- z}{\mu(\gamma,\zeta)} \frac{\zeta- \bar{z}}{\mu(\gamma,\zeta)}} \right)^{\frac{1}{2}-\nu} \\
&=
\big(\mu(\gamma,\zeta)\big)^{1-2\nu}\, \left( \frac{\abs{\im{z}}}{(\zeta- z)(\zeta- \bar{z})} \right)^{\frac{1}{2}-\nu}.
\end{align*}

Next, we assume the second case $\zeta \in \IH$ and $\gamma z \in \gamma \zeta  + i\IR_{> 0}$.
In particular, we have that $\gamma \zeta - \gamma z$ and also $\gamma \zeta - \gamma \bar{z}$ are non-vanishing purely imaginary:
\[
\gamma \zeta - \gamma z = -it
\qquad \text{and} \qquad
\gamma \zeta - \gamma z = it^\prime
\]
for some $t,t^\prime \in \IR_{\geq 0}$.
We have in fact $t^\prime = t + 2 \re{\gamma \zeta}$.
Hence the expression
\[
\frac{\abs{\im{\gamma z}}}{(\gamma \zeta- \gamma z)(\gamma \zeta- \gamma \bar{z})}
=
\frac{\abs{\im{\gamma z}}}{(-it) it^\prime }
=
\frac{\abs{\im{\gamma z}}}{t t^\prime}
\]
is positive.
On the other hand, we have
\[
\frac{\abs{\im{\gamma z}}}{(\gamma \zeta- \gamma z)(\gamma \zeta- \gamma \bar{z})}
=
\frac{\frac{\abs{\im{z}}}{\mu(\gamma,z) \, \mu(\gamma,\bar{z})}}{\frac{\zeta- z}{\mu(\gamma,\zeta)\, \mu(\gamma,z)}
  \frac{\zeta- \bar{z}}{\mu(\gamma,\zeta) \, \mu(\gamma,\bar{z})}}
= \frac{\abs{\im{z}}}{(\zeta- z) (\zeta- \bar{z})} \,
\big(\mu(\gamma,\zeta) \big)^{2}.
\]
We may apply Lemma~\ref{D2.23} to $\mu(\gamma,\zeta)^2$ and $\frac{\abs{\im{z}}}{(\zeta- z)(\zeta- \bar{z})}$ since the assumption \linebreak $\re{\mu(\gamma,\zeta)} > 0$ implies $\abs{\arg{\mu(\gamma,\zeta)}} \leq \frac{\pi}{2}$ and hence $\abs{\arg{\big(\mu(\gamma,\zeta) \big)^{2}}} < \pi$:
\begin{align*}
\left( \frac{\abs{\im{\gamma z}}}{(\gamma \zeta- \gamma z)(\gamma \zeta- \gamma \bar{z})} \right)^{\frac{1}{2}-\nu}
&=
\left( \frac{\abs{\im{z}}}{\frac{\zeta- z}{\mu(\gamma,\zeta)} \frac{\zeta- \bar{z}}{\mu(\gamma,\zeta)}} \right)^{\frac{1}{2}-\nu} \\
&=
\big(\mu(\gamma,\zeta)\big)^{1-2\nu}\, \left( \frac{\abs{\im{z}}}{(\zeta- z)(\zeta- \bar{z})} \right)^{\frac{1}{2}-\nu}.
\end{align*}

Finally, we assume the third case $\zeta \in \IH^-$ and $\gamma z \in \gamma \zeta  + i\IR_{> 0}$.
We show that \eqref{D2.13} holds by interchanging the role of $z$ and $\bar{z}$ in the calculation above.

This proves the identity \eqref{D2.13} for all three cases.

\smallskip

We need also the identity
\begin{equation}
\label{D2.25}
\left(\frac{\sqrt{\gamma \zeta - \gamma z\,}}{\sqrt{\gamma \zeta - \gamma \bar{z}\,}}\right)^{-k}
=
e^{ik\arg{\mu(\gamma,z)}} \, \left(\frac{\sqrt{\zeta- z\,}}{\sqrt{\zeta- \bar{z}\,}}\right)^{-k}.
\end{equation}
Indeed, it follows also by applying \eqref{B.4} and the assumption $\mu(\gamma,z) \in \ICprime$ as the following calculation shows:
\begin{align*}
\left(\frac{\sqrt{\gamma \zeta - \gamma z\,}}{\sqrt{\gamma \zeta - \gamma \bar{z}\,}}\right)^{-k}
&=
\left(\frac{\sqrt{\frac{\zeta- z}{\mu(\gamma,\zeta)\, \mu(\gamma,z)}\,}}{\sqrt{\frac{\zeta- \bar{z}}{\mu(\gamma,\zeta)\, \mu(\gamma,\bar{z})}\,}}\right)^{-k} \\
&=
\left(\frac{\sqrt{\mu(\gamma,\bar{z})\,}}{\sqrt{\mu(\gamma,z)\,}}\right)^{-k}
   \left(\frac{\sqrt{\zeta- z\,}}{\sqrt{\zeta- \bar{z}\,}}\right)^{-k} \\
&=
e^{ik\arg{\mu(\gamma,z)}} \, \left(\frac{\sqrt{\zeta- z\,}}{\sqrt{\zeta- \bar{z}\,}}\right)^{-k}.
\end{align*}
We also used \eqref{D2.31} and that $k \in \frac{1}{2}\IZ$ is real.

\smallskip

To finally prove the lemma we combine the identities \eqref{D2.13} and \eqref{D2.25}.
We have
\begin{align*}
&\!\!\!\!
R_{k,\nu}(\gamma z,\gamma \zeta) \\
&= \left(\frac{\sqrt{\gamma \zeta - \gamma z\,}}{\sqrt{\gamma \zeta - \gamma \bar{z}\,}}\right)^{-k}
  \, \left( \frac{\abs{\im{\gamma z}}}{(\gamma \zeta- \gamma z)(\gamma \zeta- \gamma \bar{z})} \right)^{\frac{1}{2}-\nu} \\
&=  e^{ik\arg{\mu(\gamma,z)}} \, \big(\mu(\gamma,\zeta)\big)^{1-2\nu} \;  \left(\frac{\sqrt{\zeta- z\,}}{\sqrt{\zeta- \bar{z}\,}}\right)^{-k}
  \, \left( \frac{\abs{\im{z}}}{(\zeta- z)(\zeta- \bar{z})} \right)^{\frac{1}{2}-\nu} \\
&=  e^{ik\arg{\mu(\gamma,z)}} \, \big(\mu(\gamma,\zeta)\big)^{1-2\nu} \; R_{k,\nu}(z,\zeta).
\end{align*}
\end{proof}

\begin{remark}
\label{D2.22} The second assumption on $\zeta$ and $z$ in
Lemma~\ref{D2.26} is: $\zeta \in \IH$ and $\gamma z \in \gamma \zeta
+ i\IR_{> 0}$. This is equivalent tp saying that $z$ lies in the
open geodesic ray connecting $\zeta \in \IH$ to the cusp
$\gamma^{-1}(i\infty)$. The third assumption can be rephrased
analogously: The third condition is equivalent with saying that
$\bar{z}$ lies in the geodesic ray connecting $\bar{\zeta} \in \IH$
to the cusp $\gamma^{-1}(i\infty)$.
\end{remark}

\subsection{The Maass-Selberg Differential Form}
\label{D2.33}
We recall differential forms presented in \cite{Mu03} and observe the action of Maass raising and lowering operators applied to those differential forms.

Let $f,g$ be real-analytic functions and write again $z =x+iy$. We
define
\begin{equation}
\label{D2.8} \big\{f,g\big\}^+(z) = f(z) g(z) \frac{\dd z}{y} \qquad
\text{and} \qquad \big\{f,g\big\}^-(z) = f(z) g(z) \frac{\dd \bar{z}}{y}.
\end{equation}

We extend the slash-notation to linear combinations of the differential forms $\{f,g\}^\pm$:
We define
\begin{equation}
\label{D2.7}
\begin{split}
\big\{f,g\big\}^+\big|_k^v\gamma (z)
&=e^{-ik \arg{\mu(\gamma,z)}}  \, v(\gamma)^{-1} \, f(\gamma z) g(\gamma z) \, \frac{\dd (\gamma z)}{\im{\gamma z}} \\
\big\{f,g\big\}^-\big|_k^v\gamma (z)
&=e^{-ik \arg{\mu(\gamma,z)}}  \, v(\gamma)^{-1} \, f(\gamma z) g(\gamma z) \, \frac{\dd \overline{(\gamma z)}}{\im{\gamma z}} ,
\end{split}
\end{equation}
and extend it in the obvious way to linear combinations.

\begin{lemma}
\label{D2.9} Let $v$ be a multiplier and $k, q \in \frac{1}{2}\IZ$.
\begin{enumerate}
%
\item
We have for any matrix $\gamma \in \SL{\IZ}$ that
\begin{equation}
\label{D2.10} \big\{f,g\big\}^\pm \big|_{k + q \mp 2}^v \gamma =
\begin{cases}
\big\{f \big|_k^v \gamma, g \big|_q^1 \gamma \big\}^\pm
    & \text{and}\\
\{f\big|_k^1 \gamma, g \big|_q^v \gamma \}^\pm.
\end{cases}
\end{equation}
\item
$\big\{f,g\big\}^\pm = \big\{g,f\big\}^\pm$.
\end{enumerate}
\end{lemma}

\begin{proof}
The last property follows directly from the definition in \eqref{D2.8}.

We prove \eqref{D2.10} by direct calculation, using the identities in \eqref{B.4}.
For example, we have
\begin{align*}
\big\{f,g\big\}^+ \big|_{k + q - 2}^v (z)
&=
v(\gamma)^{-1}\, e^{(-k - q + 2) i\arg{\mu(\gamma,z)}} \, f(\gamma z) \, g(\gamma z) \, \frac{\dd (\gamma z)}{\im{\gamma z}} \\
&=
v(\gamma)^{-1} \, e^{(-k-q) i\arg{\mu(\gamma,z)}} \, f(\gamma z) \, g(\gamma z) \, \frac{\dd z}{\im{z}} \\
&=
\big\{f\big|_k^v \gamma, g \big|_q^1 \gamma \big\}^+(z)
\end{align*}
for all $z \in \IH$.
The other identities follow analogously.
\end{proof}

Combining the property \eqref{D2.10} of the $1$-forms
in \eqref{D2.8} with Maass operators, we see that
\begin{equation}
\label{D2.12}
\big\{ \EE^\pm_k \big(f\big|_k^v\gamma \big) ,\,g\big|_q^1\gamma
\big\}^\pm = \big\{ \EE^\pm_k \big(f\big|_k^1\gamma \big)
,\,g\big|_q^v \gamma \big\}^\pm = \big\{ \EE^\pm_k f, \,g \big\}^\pm
\, \big|_{k + q}^v \gamma
\end{equation}
for every $\gamma \in \SL{\IR}$ and $k,q \in \frac{1}{} \IZ$. We also have the relations
\begin{equation}
\label{D2.11}
\begin{split}
\big\{ \EE^+_k f,\,g \big\}^+ &= - \big\{  f,\,\EE^+_{-k} g \big\}^+ +4i
\partial_z(fg)\dd z
\quad\text{and}\\
\big\{ \EE^-_k f,\,g \big\}^- &= - \big\{  f,\,\EE^-_{-k} g \big\}^- - 4i
\partial_{\bar{z}}(fg)\dd \bar{z}.
\end{split}
\end{equation}
and
\begin{align*}
\big\{ \EE^-_k f,\,g \big\}^+ &= - \big\{  f,\,\EE^-_{-k} g \big\}^+ -4i
\partial_{\bar{z}}(fg)\dd z
\quad \text{and} \\
\big\{ \EE^+_k f,\,g \big\}^- &=
 - \big\{  f,\,\EE^+_{-k} g \big\}^- + 4i \partial_z(fg)\dd \bar{z}.
\end{align*}

We are now able to define the Maass-Selberg form.
\begin{definition}
\label{D2.14} Let $f,g$ be real-analytic and $k \in \frac{1}{2}\IZ$. We define
the \emph{Maass-Selberg form} $\eta_k$ by
\begin{equation}
\label{D2.15} \eta_k(f,g) = \big\{\EE^+_{k} f,g\big\}^+ -
\big\{f,\EE^-_{-k} g\big\}^-.
\end{equation}
\end{definition}

\begin{lemma}
\label{D2.16} Let $f,g$ be real-analytic and $k \in \frac{1}{2}\IZ$. The
Maass-Selberg form has the following properties:
\begin{enumerate}
\item
We have the equations
\begin{equation}
\label{D2.17} \eta_k(f,g) + \eta_{-k}(g,f) = 4i \cdot \mathrm{d}
(fg)
\end{equation}
and
\[
\eta_k(f,g) - \eta_{-k}(g,f) =  4\big( [gf_y-fg_y]\mathrm{d}x +
[fg_x-gf_x+\frac{ik}{y} fg]\mathrm{d}y \big)
\]
where $f_x$ denotes $\partial_x f$, $f_y=\partial_y f$,
$g_x=\partial_x g$ and $g_y=\partial_y g$, respectively.
\item
If there exists a $\lambda \in \IR$ such that $f$ and $g$ satisfy
$\Delta_k f = \lambda f$ and $\Delta_{-k} g = \lambda g$, then the
Maass-Selberg form is closed.
\item
We have that
\begin{equation}
\label{D2.18}
\eta_k(f,g)\big|_0^v \gamma = \eta_k \big( f \big|_k^v \gamma,   g \big|_{-k}^1 \gamma \big) = \eta_k \big( f \big|_k^1 \gamma,   g \big|_{-k}^v \gamma \big)
\end{equation}
for any multiplier system $v$.
\item
Let $\nu \in \IC$ and assume that $f$ and $g$ are eigenfunctions of
the operators $\Delta_k$ and $\Delta_{-k}$, respectively, both with
eigenvalue $\frac{1}{4}-\nu^2$. Then, we have that
\begin{equation}
\label{D2.19}
\begin{split}
\eta_{k+2}(\EE^+_kf ,\, \EE^-_{-k}g)
&=
(1+2\nu+k)(1-2\nu+k) \, \eta_k(f,\, g)\\
&\quad
+ 4i \, \dd\left( (\EE^+_kf) (\EE^-_{-k}g) \right).
\end{split}
\end{equation}
\end{enumerate}
\end{lemma}

\begin{proof}
Recall that
$\partial_z =\frac{1}{2}\partial_x - \frac{i}{2} \partial_y$,
${ \partial}_{\bar{z}} =\frac{1}{2}\partial_x + \frac{i}{2} \partial_y$,
$\mathrm{d}z=\mathrm{d}x+i\mathrm{d}y$,
$\mathrm{d}\bar{z} = \mathrm{d}x - i\mathrm{d}y$
and $\dd f = \partial_{z} f \dd z + \partial_{\bar{z}} f \dd \bar{z}$
for any function $f$ smooth in $x$ and $y$.
\begin{enumerate}
\item
We prove the first item via direct computation:
\begin{align*}
\eta_k(f,g)+\eta_{-k}(g,f)
&=
-4i\big[(fg_z+gf_z\big)\dd z +
\big(fg_{\bar{z}}+gf_{\bar{z}}\big) \dd \bar{z} \big]\\
&=
4i\dd (fg)
\end{align*}
and
\begin{align*}
\eta_k(f,g)-\eta_{-k}(g,f)
&=
4i\bigg[\big(gf_z-fg_z-\frac{ik}{2y}fg \big)\dd z \\
& \qquad
 -\big( gf_{\bar{z}}-fg_{\bar{z}} - \frac{ik}{2y}fg\big)
\dd \bar{z} \bigg]\\
&=
4\bigg[\big(gf_y-fg_y\big)\dd x +
 \big(fg_x-gf_x+i\frac{k}{y}fg\big)\dd y  \bigg],
\end{align*}
where we used \eqref{D1.5} for the Maass operators.
\item
We want to show that $\dd \,\eta_k(f,g) =0$ under the given
conditions. Since
\[
2 \eta_k(f,g) =
\big(  \eta_k(f,g) + \eta_{-k}(g,f) \big) +
\big(  \eta_k(f,g) - \eta_{-k}(g,f) \big),
\]
and since we already proved the first part of the lemma,
it is enough to show that $\eta_k(f,g)-\eta_{-k}(g,f)$ is closed.
We find, after some computation, that
\[
\dd \big(\eta_k(f,g)-\eta_{-k}(g,f)\big)
=
\big[f \Delta_{-k} g - g \Delta_k f \big]
\frac{\dd x \wedge \dd y}{y^2}.
\]
\item
This follows directly from \eqref{D2.12}.
\item
It follows directly from the equations in \eqref{D2.11} that
\begin{align*}
\eta_{k+2}(\EE^+_kf ,\, \EE^-_{-k}g)
&=
4i \dd\left[ (\EE^+_kf) (\EE^-_{-k}g) \right]\\
&\quad
- \{\EE^+_kf, \, \EE^+_{-k-2}\EE^-_{-k} g\}^+
   + \{\EE^-_{k+2}\EE^+_k f, \, \EE^-_{-k} g\}^+.
\end{align*}
We may apply \eqref{D1.3}, since we assume that $f$ and $g$ are
eigenfunctions of the Laplace operators $\Delta_k$ and
$\Delta_{-k}$,respectively, with the same eigenvalue. The statement
in the lemma follows.
\end{enumerate}
\end{proof}

\begin{remark}
\label{D2.20} The first three items of Lemma~\ref{D2.16} are
generalizations of the lemma given in \cite[II.2]{LZ01}. We have
that
\[
\eta_0(f,g) = [f,g]
\]
where $[\cdot,\cdot]$ is defined in \cite[Chapter II,
\S2]{LZ01}. Our form $\{ \cdot, \cdot\}^\pm$ in~\eqref{D2.8} differs
from the form $\{ \cdot, \cdot\}$ in \cite[II.\S2, (2.5)]{LZ01},
contrary to the notational resemblance.
\end{remark}

\begin{lemma}
\label{D2.35} Let $f$ and $g$ be smooth functions (in $x$ and $y$)
on $\IH \cup \IH^-$ satisfying $f(z) = f(\bar{z})$ and $g(z) =
g(\bar{z})$.
\begin{enumerate}
\item
We can extend the Maass-Selberg form to smooth (in $x$ and $y$) functions $f,g$
defined on the lower half-plane.
\item
The Maass-Selberg form satisfies
\begin{equation}
\label{D2.36} \eta_k(f,g) \, (z) = \eta_k(g,f) \, (\bar{z}).
\end{equation}
\end{enumerate}
\end{lemma}

\begin{proof}
\begin{enumerate}
\item
All differentials and other components used in the definition of the Maass-Selberg-form
are well defined for smooth (in $x$ and $y$) functions on the lower half-plane $\IH^-$.
Hence the extension makes sense.
\item
This follows by direct calculations. First use the relations
\eqref{D2.15}, \eqref{D2.8} and \eqref{D1.4} followed by
\eqref{D1.5} to rewrite everything depending on the pair
$(z,\bar{z})$ respectively $(x,y)$. Then, use the substitution
$(z,\bar{z})\mapsto (\bar{z},z)$ respectively $(x,y) \mapsto
(x,-y)$. As final step, use the above relations in reverse order.
\end{enumerate}
\end{proof}

\subsection{Everything Combined}
\label{D2.34}
We will now insert the function $R_{k,\nu}$ into the Maass-Selberg form and use the form to define nearly periodic
functions.

\begin{lemma}
\label{D2.28} Let $v$ be a multiplier which is compatible with the
half-integral weight $k$, $\nu \in \IC$, $u$ a Maass cusp form with weight
$k$ multiplier $v$ and eigenvalue $\frac{1}{4}-\nu^2$. Moreover let
$\gamma \in \SL{\IZ}$ and $\zeta,z \in \IC$ satisfying the assumptions of Lemma~\ref{D2.26}.

It follows that
\begin{equation}
\label{D2.29}
\begin{split}
\eta_{k}(u,R_{-k,\nu}\big(\cdot,\gamma \zeta)\big) \big|_0^v \gamma (z)
&=
\big(\mu(\gamma,\zeta)\big)^{1-2\nu} \eta_{k}\big(u,R_{-k,\nu}(\cdot,\zeta)\big) (z)
\quad \text{and}\\
\eta_{-k}(R_{-k,\nu}\big(\cdot,\gamma \zeta),u\big) \big|_0^v \gamma (z)
&= \big(\mu(\gamma,\zeta) \big)^{1-2\nu} \eta_{-k}\big(R_{-k,\nu}(\cdot,\zeta),u\big) (z).
\end{split}
\end{equation}
\end{lemma}

\begin{proof}
We show the second identity:
\begin{align*}
& \!\!\!\!
\eta_{-k}\big(R_{-k,\nu}(\cdot,\gamma \zeta),u\big) \big|_0^v \gamma (z) \\
&=
\eta_{-k}\big(R_{-k,\nu}(\cdot,\gamma \zeta) \big|_{-k}^1 \gamma, u\big|_k^v\gamma  \big) (z)
\qquad \text{using \eqref{D2.18}} \\
&=
\eta_{-k}\big(R_{-k,\nu}(\cdot,\gamma \zeta) \big|_{-k}^1 \gamma, u \big) (z) \\
&=
\big(\mu(\gamma,\zeta) \big)^{1-2\nu} \, \eta_k\big( R_{-k,\nu}(\cdot,\zeta) ,u \big) (z) \qquad \text{using \eqref{D2.27}}.
\end{align*}
The use of the transformation formula \eqref{D2.27} in the calculation above is allowed since $z$ and $\zeta$ satisfy the assumptions of Lemma~\ref{D2.26} and since $\EE^+_{- k} R_{- k,\nu}\big(z, \zeta)$ appearing in the construction of $\eta_{\pm k}\big(R_{-k,\nu}(\cdot,\zeta),u\big)(z)$ satisfies \eqref{D2.6}.

The first identity follows by the same arguments.
\end{proof}

\section{Nearly Periodic Functions}
\label{D3}
Let $u$ be a Maass cusp form of weight $k$, multiplier
$v$, and spectral value $\nu$ as defined in and below
Definition~\ref{D1.1}.

We define on $\IC\smallsetminus \IR \to \IC; $

\begin{equation}
\label{D3.1}
\begin{split}
\zeta &\mapsto f(\zeta) :=
\begin{cases}
\displaystyle
  \int_\zeta^{i\infty} \eta_{-k}\big(R_{-k,\nu}(\cdot,\zeta),u\big) (z)
    &\text{ if $\im{\zeta} > 0$ and} \\
\displaystyle
- \int_{\zeta}^{-i\infty} \eta_{k}\big(R_{-k,\nu}(\cdot,\zeta), \tilde{u} \big) (z)
    &\text{ if $\im{\zeta} < 0$},
\end{cases}
\end{split}
\end{equation}
where $\tilde{u}(z) = u(\bar{z})$ as defined in Lemma~\ref{D1.8}.
The path of integration is the geodesic ray connecting $\zeta$ and the
cusp $i\infty$ respectively $-i\infty$ in the upper respectively lower half-plane.

\begin{remark}
\label{D3.23}
\begin{enumerate}
\item
We have made a choice in Definition~\ref{D3.1} between integrating over the forms
\[
\text(1a) \quad
-\eta_{k}\big(u,R_{-k,\nu}(\cdot,\zeta)\big)
\quad \text{or} \quad \text(1b) \quad
\eta_{-k}\big(R_{-k,\nu}(\cdot,\zeta),u\big)
\]
on $\IH$ and
\[
\text(2a) \quad
-\eta_{k}\big(R_{-k,\nu}(\cdot,\zeta), \tilde{u}\big)
\quad \text{or} \quad \text(2b) \quad
\eta_{-k}\big(\tilde{u},R_{-k,\nu}(\cdot,\zeta)\big)
\]
on $\IH^-$.

We will see later in Remark~\ref{D4.14} that each choice leads to
the same period function on $\IH$ prospectively $\IH^-$.
Remark~\ref{D3.24} compares our choice to the situation discussed in
\cite{LZ01}.
\item
The reason, why we extended Maass cusp forms and the $R$-function to the lower half-plane $\IH^-$ and also extended the Maass-Selberg-form $\eta_k$ to functions on the lower half-plane, see e.g.\ Lemma~\ref{D1.8}, Proposition~\ref{D2.21} and Lemma~\ref{D2.35} respectively, is the second case of \eqref{D3.1}.
We want to be able to integrate along the geodesic ray $\zeta -i\IR_{>0}$ connecting $\zeta \in \IH^-$ and $-i\infty$ in the lower half-plane.
In our opinion, this representation illustrates better how the function $f$ is defined on $\IH^-$ compared to the (on $\IH^-$ equivalent) integral representation in Lemma~\ref{D3.25}.
\end{enumerate}
\end{remark}

\begin{lemma}
\label{D2.37} For $\abs{\re{\nu}}< \frac{1}{2}$, the integration in
\eqref{D3.1} is well-defined along the geodesic paths connecting
$\zeta$ to $i\infty$ in $\IH$ respectively $\zeta$ to $-i\infty$ in
$\IH^-$.
\end{lemma}

\begin{proof}
The singularity of $R_{k,\nu}(z,\zeta)$ for $z \to \zeta \in \IH$
respectively $z \to \bar{\zeta} \in \IH$ is of the form
$(\zeta-z)^{\nu-\frac{1}{2}}$ respectively
$(\zeta-\bar{z})^{\nu-\frac{1}{2}}$. The whole integrand has at most
the same singularity since $R_{k,\nu}$ is an eigenfunction of the
Maass operators, see \eqref{D2.6}. The weight argument
$\arg{\zeta-z}$ has also a well defined limit. Hence the integration
is well defined for $\nu$ values satisfying $\abs{\re{\nu}} <
\frac{1}{2}$.
\end{proof}

\begin{lemma}
\label{D3.25} For $\zeta \in \IH^-$ we have
\begin{equation}
\label{D3.26}
f(\zeta)
=
-\int_{\bar{\zeta}}^{i\infty} \eta_{k}\big(u,R_{-k,\nu}(\cdot,\zeta)\big) (z).
\end{equation}
\end{lemma}

\begin{proof}
For $\zeta \in \IH^-$ we have
\begin{align*}
f(\zeta) &=
-\int_{\zeta}^{-i\infty} \eta_{k}\big(R_{-k,\nu}(\cdot,\zeta),\tilde{u} \big) (z) \\
&=
-\int_{\zeta}^{-i\infty} \eta_{k}\big(u, R_{-k,\nu}(\cdot,\zeta)\big) (\bar{z}) \qquad \text{using \eqref{D2.36}}\\
&=
-\int_{\bar{\zeta}}^{i\infty} \eta_{k}\big(u, R_{-k,\nu}(\cdot,\zeta)\big) (z),
\end{align*}
where we used $u(z)=\tilde{u}(\bar{z})$ for $z \in \IH$ in Lemma~\ref{D1.8}.
\end{proof}

\begin{remark}
\label{D3.24} For $k=0$, we can compare the definition of $f$ in
\eqref{D3.1} to the one in \cite[page~212]{LZ01} since we have
$\eta_0(f,g) = [f,g]$, see Remark~\ref{D2.20}. We find that
\cite{LZ01} uses exactly the opposite choice: They use
$\int_\zeta^{i\infty} \eta_{0}\big(u,R_{-0,\nu}(\cdot,\zeta)\big)$
for $\zeta \in \IH$ (compared to $\int_\zeta^{i\infty}
\eta_{-0}\big(R_{-0,\nu}(\cdot,\zeta),u\big)$ in \eqref{D3.1}). On
the lower half plane they use $-\int_{\bar{\zeta}}^{i\infty}
\eta_{-0}\big(R_{-0,\nu}(\cdot,\zeta), u \big)$ for $\zeta \in
\IH^-$ (compared to $-\int_{\bar{\zeta}}^{i\infty}
\eta_{0}\big(u, R_{-0,\nu}(\cdot,\zeta)\big)$ in \eqref{D3.26}).
\end{remark}

\smallskip

In the following lemma, we describe the transformation property of
the function $f(\zeta)$.
\begin{lemma}
\label{D3.3} Let $\abs{\re{\nu}} < \frac{1}{2}$.
The function $f$ defined in \eqref{D3.1} satisfies
\begin{equation}
\label{D3.4}
\begin{split}
f \big\|_{\nu}^v \gamma (z)
&= v(\gamma)^{-1} \, \big(\mu(\gamma,\zeta)\big)^{2\nu-1} \,
f(\gamma \, \zeta) \\
&=
\begin{cases}
\displaystyle  \int_{\zeta}^{\gamma^{-1}i\infty}
\eta_{-k}\big(R_{-k,\nu}(\cdot,\zeta),u\big) ( z)
&\text{if $\zeta \in \IH$ and}\\
\displaystyle -\int_{\overline{\zeta}}^{\gamma^{-1}i\infty}
\eta_{k}\big(u, R_{-k,\nu}(\cdot,\zeta)\big) ( z)
&\text{if $\zeta \in \IH^-$}
\end{cases}
\end{split}
\end{equation}
for every $\zeta \in \IH \cup \IH^-$ and $\gamma \in \SL{\IZ}$ satisfying $\re{\mu(\gamma,\zeta)} >0$.
The path of integration on the right hand side is the geodesic ray connecting $\zeta$ respectively $\bar{\zeta}$ and $\gamma^{-1}(i\infty)$.
\end{lemma}

\begin{proof}
Let $\zeta \in \IH$ and $\gamma \in \SL{\IZ}$ such that $\re{\mu(\gamma,\zeta)} >0$ holds.
We get
\begin{align*}
f(\gamma \, \zeta)
&= \int_{\gamma \zeta}^{i\infty}
\eta_{-k}\big(R_{-k,\nu}(\cdot,\gamma \zeta),u\big) (z)
=
\int_{\zeta}^{\gamma^{-1}i\infty} \eta_{-k}\big(R_{-k,\nu}(\cdot,\gamma \zeta),u\big) (\gamma z) \\
&= v(\gamma) \,
\int_{\zeta}^{\gamma^{-1}i\infty} \eta_{-k}\big(R_{-k,\nu}(\cdot,\gamma \zeta),u\big)\big|^v_0 \gamma (z).
\end{align*}
Now, we would like to apply Lemma~\ref{D2.28}.
Therefore, we must check, if all $z$ of the integration path satisfy the conditions on $\zeta$ and $z$ as given in Lemma~\ref{D2.26}.
Remark~\ref{D2.22} implies that we have to verify if the integration path is the geodesic ray connecting $\zeta$ and $\gamma^{-1} (i\infty)$.
This is indeed the case. Hence the second condition of Lemma~\ref{D2.26} is satisfied since we assume $\re{\mu(\gamma,\zeta)} > 0$.
Using the transformation formula \eqref{D2.29} in Lemma~\ref{D2.28} gives
\begin{align*}
f(\gamma \, \zeta)
&= v(\gamma) \,
\int_{\zeta}^{\gamma^{-1}i\infty} \eta_{-k}\big(R_{-k,\nu}(\cdot,\gamma \zeta),u\big)\big|^v_0 \gamma (z) \\
&= v(\gamma) \, \big(\mu(\gamma,\zeta)\big)^{1-2\nu} \,
\int_{\zeta}^{\gamma^{-1}i\infty}
\eta_{-k}\big(R_{-k,\nu}(\cdot,\zeta),u\big) ( z).
\end{align*}

The same calculation for $\zeta \in \IH^-$, using the integral for
$f$ in \eqref{D3.26}, gives
\[
v(\gamma)^{-1} \, \big(\mu(\gamma,\zeta)\big)^{2\nu-1} \, f(\gamma \, \zeta)
= -\int_{\overline{\zeta}}^{\gamma^{-1}i\infty}
\eta_{k}\big(u, R_{-k,\nu}(\cdot,\zeta) \big) (z).
\]
\end{proof}

\begin{definition}
\label{D3.12}
We call a function $g$ \emph{nearly periodic} if there exists an $a \in \IC$ with $\abs{a}=1$ such that $g(z+1) = a \, g(z)$ holds for all $z$.
\end{definition}

We check that $f(\zeta)$ is nearly periodic as application of
Lemma~\ref{D3.3}: For $\zeta \in \IH$ we find
\begin{align*}
v(T)^{-1} \, f(T\zeta) &=
\int_{\zeta}^{T^{-1}i\infty} \eta_{-k}\big(R_{-k,\nu}(\cdot,\zeta),u\big) (z) \\
&=
\int_{\zeta}^{i\infty} \eta_{-k}\big(R_{-k,\nu}(\cdot,\zeta),u\big) (z)\\
&= f(\zeta).
\end{align*}
where we use the invariance of the cusp $i\infty$ under translation and the trivial fact that $\re{\mu(T,\zeta)}=1$.
The same arguments hold for $\zeta \in \IH^-$. Hence, we just proved
the following
\begin{lemma}
\label{D3.5} The function $f$ in defined in \eqref{D3.1} satisfies
\begin{equation}
\label{D3.2}
 v(T)^{-1} \, f(\zeta +1) = f(\zeta)    \qquad \text{for every $\zeta \in \IC \smallsetminus \IR$}.
\end{equation}
Written in the double-slash notation~\eqref{B.12}, we have
\begin{equation}
\label{D3.11} f\big\|_{\nu}^v T = f     \qquad \text{on $\IC
\smallsetminus \IR$}.
\end{equation}
\end{lemma}

Similar to \cite[Proposition~2]{LZ01} we continue to prove an algebraic
correspondence between $f$ and a solution of a suitable three-term
equation on $\IC \smallsetminus \IR$.

\begin{lemma}
\label{D3.7} Assume that $k$ and $\nu$ satisfy $e^{\mp \pi
i(2\nu-1)} \neq e^{\pi ik}$. Put
\begin{equation}
\label{D3.8} c^\star_\pm = 1 - e^{\pi ik} e^{\pm \pi i(2\nu-1)}.
\end{equation}
Then, there exists a bijection between nearly periodic functions $f$
satisfying \eqref{D3.2} and solutions $P$ of the \emph{three-term equation}
\begin{align}
\label{D3.16} &
P(\zeta) = v(T)^{-1} P(\zeta +1) +  v (T^\prime)^{-1} (\zeta+1)^{2\nu-1} P\left( \frac{\zeta}{\zeta+1} \right)\\
\nonumber \text{i.e.}&\ \ \  P\big\|_{\nu}^v \big(\id - T -
T^\prime \big)(\zeta)=0\ \ \text{for} \ \ \ \zeta \in \IC
\smallsetminus \IR.
\end{align}
The bijection is given by the formulas:
\begin{align}
\label{D3.9} c^\star_\pm \, f(\zeta) &= P(\zeta) +  v(S)^{-1} \,
 \zeta^{2\nu-1} \, P\big(S\zeta \big)
    \qquad (\im{\zeta} \gtrless 0) \\
\nonumber &=
P\big\|_{\nu}^v \big(\id + S\big) (\zeta)\\
\intertext{and} \label{D3.10} P(\zeta) &= f(\zeta) - v(S)^{-1} \,
\zeta^{2\nu-1} \, f\big(S\zeta \big)
    \qquad (\zeta \in \IC \smallsetminus \IR) \\
\nonumber &= f\big\|_{\nu}^v \big(\id - S\big) (\zeta).
\end{align}
\end{lemma}

\begin{remark}
\label{D3.19} Observe at a formal level that $P$, as defined in
\eqref{D3.10}, satisfies the three-term functional
equation~\eqref{D3.16}:
\begin{align*}
& \!\!\!\!
P \big\|_{\nu}^v (\id - T - T^\prime) \\
&=
f \big\|_{\nu}^v (\id-S)
\big\|_{\nu}^v (\id - T - TST)
\qquad \qquad (\text{$TST= T^\prime$ by \eqref{B.2}}) \\
&=
f \big\|_{\nu}^v (\id - S - T + ST - TST + STST) \\
&= f \big\|_{\nu}^v (\id - S - T + ST - TST + T^{-1} S)
\qquad (\text{$STST = T^{-1} S$ by \eqref{B.6}}) \\
&=
f \big\|_{\nu}^v \big((\id - T) + (T^{-1} S - S) + (ST - TST)\big) \\
&=
f \big\|_{\nu}^v (\id - T) \big\|_{\nu}^v (\id + T^{-1} S + ST) \\
&= 0.
\end{align*}
However, the calculation is only formal, since the double-slash notation
just hides the weight factors and the multipliers. In general, we do
not know whether they match since the double-slash notation is not a group
action. We have to check them on each occasion.
\end{remark}

\begin{proof}[Proof of Lemma~\ref{D3.7}]
%
Let $z \in \IC\smallsetminus \IR$. First, we compute
\[
v(S)^{-1}v(S)^{-1} \, \zeta^{2\nu-1} \big(\frac{-1}{\zeta}\big)^{2\nu-1}.
\]
We have
\[
\zeta^{2\nu-1}\left(\frac{-1}{\zeta}\right)^{2\nu-1} =
e^{(2\nu-1)i\big(\arg{\zeta} + \arg{\frac{-1}{\zeta}}\big)} = e^{\pm \pi
i(2\nu-1)} \quad(\im{\zeta} \gtrless 0)
\]
since $\arg{\zeta} + \arg{-\frac{1}{\zeta}} = \pm \pi$ for $\im{\zeta} \gtrless
0$. The choices $+$ and $>$, respectively $-$ and $<$ correspond.
The consistency relation \eqref{B.10} for multipliers implies
\[
v(S) v(S) = e^{-ik\pi}.
\]
Hence
\begin{equation}
\label{D3.6}
v(S)^{-1}v(S)^{-1} \, \zeta^{2\nu-1} \big(\frac{-1}{\zeta}\big)^{2\nu-1} \\
=
e^{\pi ik} e^{\pm \pi i(2\nu-1)}
\end{equation}
holds.

Next, we show that \eqref{D3.9} and \eqref{D3.10} are inverses of
each other. On one hand, we have
\begin{align*}
c^\star_\pm  f(\zeta) &=
P(\zeta) + v(S)^{-1} \zeta^{2\nu-1} P\big(\frac{-1}{\zeta} \big)\\
&=
f(\zeta) - v(S)^{-1}  \zeta^{2\nu-1} f\big(\frac{-1}{\zeta} \big)\\
& \quad
+ v(S)^{-1} \zeta^{2\nu-1} \Big[ f\big(\frac{-1}{\zeta}\big) - v(S)^{-1} \big(\frac{-1}{\zeta}\big)^{2\nu-1} f(\zeta) \Big]\\
&=
f(\zeta) \Big[ 1-v(S)^{-1}v(S)^{-1} \zeta^{2\nu-1} \big(\frac{-1}{\zeta}\big)^{2\nu-1}\Big] \\
&=
f(\zeta) \big[ 1 - e^{\pi ik} e^{\pm \pi i(2\nu-1)} \big] \qquad
(\text{for $\im{\zeta} \gtrless 0$}).
\end{align*}

On the other hand, we have
\begin{align*}
c^\star_\pm P(\zeta) &=
c^\star_\pm \big[ f(\zeta) - v(S)^{-1} \zeta^{2\nu-1} f\big(\frac{-1}{\zeta} \big) \Big]\\
&=
P(\zeta) + v(S)^{-1} \zeta^{2\nu-1} P\big(\frac{-1}{\zeta}\big) \\
&\quad
- v(S)^{-1} \zeta^{2\nu-1} \Big[ P\big(\frac{-1}{\zeta}\big) + v(S)^{-1}  \big(\frac{-1}{\zeta}\big)^{2\nu-1} P(\zeta) \Big]\\
&= P(\zeta) \big[ 1 - e^{\pi ik} e^{\pm \pi i(2\nu-1)} \big] \qquad
(\text{for $\im{\zeta} \gtrless 0$}),
\end{align*}
using the same argument calculations as above.

\smallskip

We now show, that $f$ being nearly periodic corresponds to $P$ satisfying the three-term equation.

Let $f$ be a nearly periodic function satisfying \eqref{D3.2} and $P$ function given by \eqref{D3.10}.
Then, we find (for $\zeta \in \IH \cup \IH^-$)
\begin{align*}
& 
P \big\|_{\nu}^v \big(\id - T - T^\prime\big) (\zeta) \\
&=
P(\zeta) - v(T)^{-1} \, P(T\zeta) - v(T^\prime)^{-1} (\zeta+1)^{2\nu-1} \, P(T^\prime \zeta) \\
&=
\bigg( f(\zeta) - v(S)^{-1} \zeta^{2\nu-1} \, f(S\zeta) \bigg) \\
&\quad
- v(T)^{-1} \,\bigg( f(T\zeta) - v(S)^{-1} (T\zeta)^{2\nu-1} \, f(ST\zeta) \bigg) \\
&\quad
- v(T^\prime)^{-1} (\zeta+1)^{2\nu-1} \, \cdot\\
&\qquad\qquad\qquad\cdot \,\bigg( f(T^\prime\zeta) - v(S)^{-1}  (T^\prime\zeta)^{2\nu-1} \, f(ST^\prime\zeta) \bigg) \\
&=
\bigg( f(\zeta) - v(T)^{-1} \, f(T\zeta) \bigg) \\
&\quad
+ \bigg( v(T^\prime)^{-1} v(S)^{-1}\, (\zeta+1)^{2\nu-1}  (T^\prime\zeta)^{2\nu-1} \, f(STST\zeta) \\
&\qquad\qquad\qquad\qquad
- v(S)^{-1} \, \zeta^{2\nu-1} \, f(S\zeta) \bigg) \\
&\quad
+ \bigg( v(T)^{-1} v(S)^{-1} \, (T\zeta)^{2\nu-1} \, f(ST\zeta) \\
&\qquad\qquad\qquad\qquad
 -  v(T^\prime)^{-1} \, (\zeta+1)^{2\nu-1} \, f(TST\zeta)  \bigg) \\
\intertext{ ($f$ is nearly periodic and $STST=T^{-1}S$)}
&=
0  +   v(S)^{-1} \, \zeta^{2\nu-1} \cdot\\
&\quad \cdot \,
\bigg( v(T^\prime)^{-1} \, \frac{1}{\zeta^{2\nu-1}}  (\zeta+1)^{2\nu-1} (T^\prime\zeta)^{2\nu-1}f(T^{-1}\,S\zeta)  - f(S\zeta) \bigg) \\
&\quad + v(T) v(T^\prime)^{-1} \,
(\zeta+1)^{2\nu-1}
\bigg( v(T^\prime) v(T)^{-2} v(S)^{-1} f(ST\zeta) \\
&\qquad\qquad\qquad\qquad
 -  v(T)^{-1} \, f(T \, ST\zeta)  \bigg) \\
&=
v(S)^{-1} \, \zeta^{2\nu-1} \bigg( v(T)\, f(T^{-1}\,S\zeta) - f(S\zeta)\bigg) \\
&\quad + v(T) v(T^\prime)^{-1} \, (\zeta+1)^{2\nu-1}
\bigg( f(ST\zeta)  -  v(T)^{-1} \, f(T \, ST\zeta) \bigg) \\
&= 0.
\end{align*}
We used several times multiplier identities based on the consistency relation~\eqref{B.10}.
Hence, $P$ satisfies the three-term equation~\eqref{D3.16}.

\smallskip

Conversely, let us assume that the function $P$ satisfies the
three-term equation~\eqref{D3.16} on $\IC \smallsetminus \IR$.
We have to show that $f$ attached by \eqref{D3.9} is indeed nearly periodic.
Applying the three-term equation to $P$ in $\zeta$ and
$ST\,\zeta=\frac{-1}{\zeta +1}$ we obtain:
\begin{align*}
0 &=
 \left(P\big\|_{\nu}^v \Big[ -\id + T + T^\prime\Big]\right)\big\|_{\nu}^v\big[\id - ST\big] (\zeta)\\
&=
\bigg[ -P(\zeta) + P\big\|_{\nu}^v T (\zeta) + P\big\|_{\nu}^v T^\prime (\zeta)\bigg] \; - \, v(ST)^{-1} \, (\zeta +1)^{2\nu-1} \cdot \\
&\qquad \qquad
\cdot \bigg[ -P\big(\frac{-1}{\zeta +1}\big) + P\big\|_{\nu}^v T \big(\frac{-1}{\zeta +1}\big) + P\big\|_{\nu}^v T^\prime \big(\frac{-1}{\zeta +1}\big)\bigg]\\
&=
\bigg[ -P(z) + v(T)^{-1}\,  P(\zeta +1) + v(T^\prime)^{-1} \, (\zeta +1)^{2\nu-1} \, P\big(\frac{\zeta}{\zeta +1}\big)\bigg]\\
&\quad
- v(ST)^{-1} \, (\zeta +1)^{2\nu-1} \bigg[ - P\big(\frac{-1}{\zeta +1}\big) + v(T)^{-1} P\big(\frac{\zeta}{\zeta +1}\big) + \\
&\qquad\qquad
+v(T^\prime)^{-1}  \bigg(\frac{-1}{\zeta +1}+1\bigg)^{2\nu-1} P\bigg(\frac{\frac{-1}{\zeta +1}}{\frac{-1}{\zeta +1}+1}\bigg)\bigg]\\
&=
- \bigg[ P(\zeta) + v(ST)^{-1}v(T^\prime)^{-1} \, (\zeta +1)^{2\nu-1} \bigg(\frac{\zeta}{\zeta +1}\bigg)^{2\nu-1}  \, P\big(\frac{-1}{\zeta}\big)\bigg] \\
&\quad
+\left[ v(T)^{-1} P(\zeta +1) + v(ST)^{-1} \, (\zeta +1)^{2\nu-1} P\big(\frac{-1}{\zeta +1}\big) \right] \\
&\quad
+\bigg[ v(T^\prime)^{-1} \, (\zeta +1)^{2\nu-1} P (T^\prime \zeta) \,- \\
&\qquad\qquad v(ST)^{-1}v(T)^{-1} \, (\zeta +1)^{2\nu-1}  P\big(\frac{\zeta}{\zeta +1}\big) \bigg] \\
&=
- \left[ P(\zeta) + v(S)^{-1} \, z^{2\nu-1} \, P\big(\frac{-1}{\zeta}\big)\right] \\
&\quad
+v(T)^{-1} \left[P(T\zeta) + v(S)^{-1} \, (T\zeta)^{2\nu-1} P\big(\frac{-1}{T\zeta}\big) \right]
\quad + 0\\
&=
-c^\star_\pm f(\zeta)   \; + \;  v(T)^{-1} \, c^\star_\pm f(\zeta +1)
\qquad\qquad (\text{for $\im{\zeta} \gtrless 0$}),
\end{align*}
using again multiplier identities derived from the consistency
relation~\eqref{B.10}. This shows that if $P$ satisfies the
three-term equation then $f$ is nearly periodic.
\end{proof}

\section{Period Functions}
\label{D4}

\subsection{Period Functions by Integral Transforms}
\label{D4.1} We follow \cite[\S2.3]{Mu03}, which is an extension of
\cite[Chapter II, \S2]{LZ01} to real weights, and define the
following integral transformation of a Maass cusp form.

\begin{definition}
\label{D4.2} Let $\zeta\in (0,\infty)$ and $\nu \in \IC$, a
compatible multiplier $v$ and a weight $k\in \frac{1}{2}\IZ$. Let
$u$ be a Maass cusp form of weight $k$, multiplier $v$ and
eigenvalue $\frac{1}{4}-\nu^2$.

We associate a function $P_{k,\nu}:(0,\infty) \to \IC$; $\zeta \mapsto
P_{k,\nu}(\zeta)$ to the cusp form $u$ by the integral transform
\begin{equation}
\label{D4.3}
P_{k,\nu}(\zeta)
=
\int_0^{i\infty} \eta_{-k}\big(R_{-k,\nu}(\cdot,\zeta),u\big) (z)
\end{equation}
where the path of integration is the upper imaginary axis, i.e., the
geodesic connecting $0$ and $i\infty$.
\end{definition}

The integral transform above is well defined, as the following arguments show.
Let $\zeta \in (0,\infty)$ and consider the function $R_{-k,\nu}(z,\zeta)$.
The construction of $R_{-k}(\cdot,\zeta)$ implies polynomial growth for $\im{z} \to
\infty$ and $\im{z} \downarrow 0$.
The Maass cusp form $u$ decays quicker than any polynomial at cusps, see Definition~\ref{D1.1}.
Hence, the integral $\int_0^{i\infty} \eta_{-k}\big(R_{-k,\nu}(\cdot,\zeta),u\big) (z)$ is well defined.

\begin{remark}
\label{D4.14}
The definition of $P_{k,\nu}$ in \cite[Definition~41]{Mu03} is
\[
P_{k,\nu}(\zeta)
=
\int_0^{i\infty} \eta_k\big(u,R_{-k,\nu}(\cdot,\zeta)\big) (z)
\]
which seems to differ from the one we use in \eqref{D4.3}.
However, $u$ and $R_{-k,\nu}$ are eigenfunctions of $\Delta_k$ and $\Delta_{-k}$ respectively.
This implies that the Maass-Selberg form is closed, see Lemma~\ref{D2.16}, and we have
\begin{align*}
&\int_0^{i\infty} \eta_{-k}\big(R_{-k,\nu}(\cdot,\zeta),u\big) (z) \\
&=
\int_0^{i\infty} \eta_k\big(u,R_{-k,\nu}(\cdot,\zeta)\big) (z)
+ \int_0^{i\infty} \dd \big(u(\cdot) \, R_{-k,\nu}(\cdot,\zeta) \big).
\end{align*}
Due to $u$ being cuspidal, and hence vanishing in $0$ and $i\infty$, we have
\[
\int_0^{i\infty} \dd \big(u(\cdot) \, R_{-k,\nu}(\cdot,\zeta) \big) = 0.
\]
Hence, the definitions of $P_{k,\nu}$ in \eqref{D4.3} and in
\cite[Definition~41]{Mu03} agree. This also shows that the choice
mentioned in Remark~\ref{D3.23} does not matter for the period
functions.
\end{remark}

\begin{lemma}
\label{D4.4} Let $k, v, \nu$ and $u$ be as in Definition~\ref{D4.2},
let $\zeta \in (0,\infty)$ and $\gamma \in \SL{\IZ}$ such that
$\mu(\gamma,\zeta) >0$ and $\gamma (0,\infty) \subset (0,\infty)$.
The function $P_{k,\nu}$ defined in \eqref{D4.3} satisfies
\begin{equation}
\label{D4.5}
\big(P_{k,\nu}\big\|_{\nu}^v \gamma \big)(\zeta)
=
\int_{\gamma^{-1}0}^{\gamma^{-1}\infty} \eta_{-k}\big(R_{-k,\nu}(\cdot,\zeta),u\big) (z)
\end{equation}
where the path of integration is the geodesic connecting $\gamma^{-1}0$ and $\gamma^{-1}\infty$.
\end{lemma}

\begin{proof}
We have
\begin{align*}
\big(P_{k,\nu}\big\|_{\nu}^v \gamma \big)(\zeta)
&=
v(\gamma)^{-1} \big(\mu(\gamma\zeta)\big)^{2\nu-1} \, \int_0^\infty \eta_{-k}\big(R_{-k,\nu}(\cdot,\gamma \zeta),u\big) (z) \\
&=
\int_0^\infty \eta_{-k}\big(R_{-k,\nu}(\cdot, \zeta),u\big) (\gamma^{-1} z)
\qquad \text{using Lemma~\ref{D2.28}}\\
&=
\int_{\gamma^{-1}0}^{\gamma^{-1}\infty} \eta_{-k}\big(R_{-k,\nu}(\cdot,\zeta),u\big) (z).
\end{align*}
The use of Lemma~\ref{D2.28} is valid since $\zeta$ and
$\mu(\gamma,\zeta)$ are both positive reals. The path of integration
of the last integral is the geodesic connecting $\gamma^{-1}0$ and
$\gamma^{-1}\infty$ and lies in the upper left quadrant $\{z \in
\IC;\; \re{z} \leq 0, \im{z} \geq 0\}$ of $\IC$.
\end{proof}

We show next that $P_{k,\nu}$ satisfies the \emph{three-term equation} on $\IR_+$.
\begin{lemma}
\label{D4.6} Let $\nu$, $k$ and $v$ as in Definition~\ref{D4.2} and
$u$ a Maass cusp form with weight $k$ compatible multiplier $v$ and
eigenvalue $\frac{1}{4}-\nu^2$. The function $P_{k,\nu}$ satisfies
the three-term equation
\begin{equation}
\label{D4.7}
0 = P_{k,\nu}\big\|_{\nu}^v \big(\id - T - T^\prime \big)
\qquad \text{on $(0,\infty)$}.
\end{equation}
\end{lemma}

\begin{proof}
Let $\zeta>0$. Lemma~\ref{D4.4} allows us to write
\begin{align*}
0 &=
\Big( \int_0^\infty - \int_{-1}^\infty - \int_0^{-1} \Big) \eta_{-k}\big(R_{-k,\nu}(\cdot,\zeta),u\big)(z)\\
&=
P_{k,\nu}(\zeta) - v(T)^{-1} \,P_{k,\nu}(T\zeta)  - v(T^\prime)^{-1}  (\zeta +1)^{2\nu-1} \,P_{k,\nu}( T^\prime \zeta) \\
&=
P_{k,\nu}\big\|_{\nu}^v \big(\id - T - T^\prime \big)(\zeta).
\end{align*}
\end{proof}

The next step is to extend $P_{k,\nu}(\zeta)$ to the right half
plane $\{\zeta \in \IC; \re{\zeta}>0\}$. Let $\zeta$ be in the right
half-plane and recall that $R_{-k,\nu}(z, \zeta)$ is holomorphic in
$\zeta$ if $\re{z} \leq 0$. Hence, the function $P_{k,\nu}(\zeta)$,
given by the integral transform~\eqref{D4.3} extends holomorphically
to $\{\zeta \in \IC; \re{\zeta}>0\}$. It is easily checked that
$P_{k,\nu}(\zeta+1)$ and
$P_{k,\nu}\left(\frac{\zeta}{\zeta+1}\right)$ have also holomorphic
extensions to this right half-plane.

The last step is to extend $P_{k,\nu}$ to the cut plane $\ICprime = \IC \smallsetminus (-\infty,0]$.
Assume $\re{\zeta}>0$ for the moment.
Since the differential form $\eta_k\big(u,R_{-k,\nu}(\cdot,\zeta)\big)$ is closed, see Lemma~\ref{D2.16}, we replace vertical path of integration in \eqref{D4.3} by a path which connects $0$ and $i \infty$ in the upper left quadrant and which passes to the left of either $\zeta$ or $\bar{\zeta}$.
We then may move $\zeta$ to any point for which either $\zeta$ or $\bar{\zeta}$ is still right of the new integration path.
This procedure extends $P_{k,\nu}$ to a holomorphic function on $\ICprime$.

Summarizing we have
\begin{theorem}
\label{D4.8} Under the assumptions of Definition~\ref{D4.2}, the
function $P_{k,\nu}$ associated to $u$ by \eqref{D4.3} extends to a
holomorphic function on the cut plane $\ICprime$ which satisfies the
three-term equation~\eqref{D4.7} on $\IR_{>0}$.
It also satisfies the growth conditions
\begin{equation}
\label{D4.9}
P_{k,\nu}(\zeta) =
\begin{cases}
\OO{z^{\max\{0, 2\re{\nu}-1}}
       &\text{as } \im{z}=0, \, \zeta \downarrow 0 \text{ and}\\
\OO{z^{\min\{0, 2\re{\nu}-1}}
       & \text{as } \im{z}=0, \, \zeta \to \infty.
\end{cases}
\end{equation}
\end{theorem}

\begin{proof}
The first part of the proposition follows from the discussion above.

The cusp form $u$ is bounded on $\IH$ since a cusp form vanishes at
all cusps $\IQ \cup i\infty$ and $u$ is real-analytic on $\IH$.
Also, $\EE^+_k u$ is bounded since the Maass operator maps cusp
forms of weight $k$ to cusp forms of weight $k+2$. Applying
successively \eqref{D4.3}, \eqref{D2.15}, \eqref{D2.8}, \eqref{D2.6}
and \eqref{D2.4} we find
\begin{align*}
&P_{k,\nu}(\zeta) \\
&=
\int_0^{i\infty} \bigg[ \left(\EE_{-k}^+ R_{-k,\nu}(\cdot,\zeta) \right)(z)\, u(z)  \, \frac{\dd z}{y}
\; - \;
R_{-k,\nu}(z,\zeta) \, \left(\EE_k^-u\right)(z)\, \frac{\dd \bar{z}}{y}  \bigg]\\
&=
\int_0^{i\infty} \bigg[ (1-2\nu-k)\, R_{2-k,\nu}(z,\zeta) \, u(z) \, \frac{\dd z}{y}
- R_{-k,\nu}(z,\zeta) \, \left(\EE_k^-u \right)(z)\, \frac{\dd \bar{z}}{y}  \bigg]\\
&=
i \int_0^\infty e^{ik\arg{\zeta-iy}} \, \left( \frac{y}{(\zeta-iy)(\zeta+iy)}\right)^{\frac{1}{2}-\nu} \\
& \qquad\qquad
\left[ (1-2\nu-k)\,e^{2i\arg{\zeta-iy}} u(iy)  -  \left(\EE_k^- u\right)(iy) \right] \frac{\dd y}{y}
\end{align*}
for $\zeta>0$.
Using the notation $f(z) \ll g(z)$ for $f(z) = \OO{g(z)}$, we find the estimate
\begin{align}
\label{D4.10} \abs{P_{k,\nu}(\zeta)} &\ll
\int_0^\infty  \left| \frac{y}{\zeta^2+y^2} \right|^{\frac{1}{2}-\re{\nu}}   \,\cdot\,\\
\nonumber & \qquad
  \max\bigg\{\abs{\left(\EE_k^-u\right)(iy)},  \abs{(1-2\nu-k)\,u(iy) } \bigg\} \,  \frac{\dd y}{y}.
\end{align}
for $\zeta>0$. The integral converges since $u$ and hence $u(iy)$
and $\big(E^-_k u\big)(iy)$ decay quickly as $ y \to \infty$ and as
$y \downarrow 0$.

Using the estimate
\[
\frac{y}{\zeta^2+y^2} \leq  \zeta^{-2} \,y
\]
in \eqref{D4.10} gives
\[
P_{k,\nu}(\zeta) = \OO{\zeta^{2\re{\nu}-1}}
\qquad \text{for every $\zeta>0$}.
\]
We have
\[
P_{k,\nu}(\zeta) = \OO{1}
\qquad \text{for every $\zeta >0$}
\]
if we use
\[
\frac{y}{\zeta^2+y^2} \leq y^{-1}
\]
in \eqref{D4.5}. This proves the stated growth condition.
\end{proof}

\subsection{Period Functions and Nearly Periodic Functions}
\label{D4.11} Let us start with a Maass cusp form $u$ of weight $k$,
multiplier $v$ and spectral value $\nu$ as in Definition~\ref{D1.1}.
We associated in \S\ref{D3} a nearly periodic function $f$ by the
integral transform~\eqref{D3.1}:
\[
\tag{\ref{D3.1}}
\begin{split}
\IC\smallsetminus \IR &\to \IC; \\
\zeta &\mapsto f(\zeta) :=
\begin{cases}
\displaystyle
  \int_\zeta^{i\infty} \eta_{-k}\big(R_{-k,\nu}(\cdot,\zeta),u\big) (z)
    &\text{ if $\zeta \in \IH$ and} \\
\displaystyle
 -\int_{\zeta}^{-i\infty} \eta_{k}\big(R_{-k,\nu}(\cdot,\zeta),\tilde{u} \big) (z)
    &\text{ if $\zeta \in \IH^-$.}
\end{cases}
\end{split}
\]
Then, we attached a period-function $P$ by
\eqref{D3.10}:
\[
\tag{\ref{D3.10}} P = f\big\|_{\nu}^v (\id - S) \qquad \qquad
\text{(on $\IH \cup \IH^-$)}
\]
which satisfies the three-term equation
\[
\tag{\ref{D3.16}} 0= P\big\|_{\nu}^v (\id - T-T^\prime) \qquad
\qquad \text{(on $\IH \cup \IH^-$)}.
\]

On the other hand, we have the integral transformation
\eqref{D4.3} from the Maass cusp form $u$ to the period function
$P_{k,\nu}$:
\[
\tag{\ref{D4.3}}
P_{k,\nu}(\zeta) = \int_0^{i\infty} \eta_{-k}\big(R_{-k,\nu}(\cdot,\zeta),u\big) (z)
\qquad \text{(on $\IR_{>0}$)}
\]
which satisfies the three-term equation
\[
\tag{\ref{D4.7}} 0= P\big\|_{\nu}^v (\id - T-T^\prime) \qquad
\qquad \text{(on $\IR_{>0}$)}
\]
and extends to $\ICprime$ (Theorem~\ref{D4.8}).

Are both directions compatible? In other words, do we get the same
function $P$ on $\IH \cup \IH^-$, regardless of using the
intermediate periodic function via \eqref{D3.1} and \eqref{D3.10} of
taking the formula \eqref{D4.3}?

\begin{lemma}
\label{D4.12} Let $k, v, \nu$ and $u$ be as in Definition~\ref{D4.2}
with $\abs{\re{\nu}} < \frac{1}{2}$. The maps
\[
u \stackrel{\eqref{D3.1}}{\longmapsto} f \stackrel{\eqref{D3.10}}{\longmapsto} P
\quad \text{and} \quad
u \stackrel{\eqref{D4.3}}{\longmapsto} P_{k,\nu}
\]
give rise to the same function $P=P_{k,\nu}$ on $\big\{\zeta \in \IC; \; \re{\zeta} >0, \, \im{\zeta} \neq 0 \big\}$.
\end{lemma}

\begin{proof}
For $\zeta \in \IH$ with $\re{\zeta}>0$ in the upper half-plane, we find
\begin{align*}
P_{k,\nu}(\zeta)
&=
\int_0^{i\infty} \eta_{-k}\big(R_{-k,\nu}(\cdot, \zeta),u\big) (z)  \\
&=
\int_\zeta^{i\infty} \eta_{-k}\big(R_{-k,\nu}(\cdot, \zeta),u\big) (z) +
\int_0^\zeta \eta_{-k}\big(R_{-k,\nu}(\cdot,\zeta),u\big) (z) \\
&=
\int_\zeta^{i\infty} \eta_{-k}\big(R_{-k,\nu}(\cdot, \zeta),u\big) (z) -
\int_\zeta^{S^{-1}i\infty} \eta_{-k}\big(R_{-k,\nu}(\cdot, \zeta),u\big) (z) \\
&=
f(\zeta) - v(S)^{-1} \, \, \zeta^{2\nu-1} \, f(S \, \zeta)
\qquad \text{using Lemma~\ref{D3.3}} \\
&= P(\zeta).
\end{align*}
A similar calculation holds for $\zeta \in \IH^-$ with $\re{\zeta}>0$:
\begin{align*}
P_{k,\nu}(\zeta)
&=
-\int_0^{i\infty} \eta_{k}\big(u, R_{-k,\nu}(\cdot, \zeta)\big) (z)
\qquad \text{using Lemma~\ref{D2.17}}\\
&=
-\int_{\overline{\zeta}}^{i\infty} \eta_{k}\big(u, R_{-k,\nu}(\cdot, \zeta)\big) (z)
-\int_0^{\overline{\zeta}} \eta_{k}\big(u, R_{-k,\nu}(\cdot,\zeta) \big) (z) \\
&=
f(\zeta) - v(S)^{-1} \, \zeta^{2\nu-1} \, f(S \, \zeta)
\qquad \text{using Lemma~\ref{D3.3}} \\
&= P(\zeta).
\end{align*}
\end{proof}

\begin{theorem}
\label{D4.13}
The function $P$ given by \eqref{D3.10} on $\IH \cup \IH^-$ is holomorphic, extends holomorphically to the
cut-plane $\ICprime = \IC \setminus (-\infty,0]$, satisfies the three-term-equation $0= P\big\|_{\nu}^v (\id - T-T^\prime)$ on $\ICprime$, and satisfies the growth condition~\eqref{D4.9}.
\end{theorem}

\begin{proof}
Lemma~\ref{D4.12} shows that $P$ agrees on $\big\{\zeta \in \IC; \; \re{\zeta} >0, \, \im{\zeta} \neq 0 \big\}$ with
$P_{k,\nu}$ given by \eqref{D4.3}. The latter extends
holomorphically to $\ICprime$ and satisfies the growth
condition~\eqref{D4.9} by Proposition~\ref{D4.8}.
\end{proof}

\section{Proof of Theorem~\ref{A.1}}
\label{G}
Our main theorem is basically proven in \S\ref{D3} and \S\ref{D4}.
We just have to collect all parts.

The map $u \mapsto f$ from Maass cusp forms to nearly periodic functions is defined in \eqref{D3.1}.
That $f$ is nearly periodic is shown in Lemma~\ref{D3.5} and Theorem~\ref{D4.13} shows the remaining part.

The bijection $f \leftrightarrow P$ is due to Lemma~\ref{D3.7}.

The map $u \mapsto P_{k,\nu}$ from Maass cusp forms to period functions is given in \eqref{D4.3}.
The properties of $P_{k,\nu}$ are described in Theorem~\ref{D4.8}.

Lemma~\ref{D4.12}, cumulating in Theorem~\ref{D4.13}, shows that the period function $P$ obtained via $u \stackrel{\eqref{D3.1}}{\mapsto} f \stackrel{\text{L}~\ref{D3.7}}{\mapsto} P$ and via $u \stackrel{\eqref{D4.2}}{\mapsto} P_{k,\nu}$ are the same.

This concludes the proof of Theorem~\ref{A.1}.

\section{Period Functions and Period Polynomials}
\label{E}
In the following section we compare the integral transformation \eqref{D4.3} and the classical Eichler
integral in \eqref{C2.1} for holomorphic cusp forms.

\smallskip

Let $u_\text{h}$ be a modular cusp form of weight $k \in 2\IN$ as defined in the introduction.
We attach a Maass cusp form $u:\IH \to \IC$ to $u_\text{h}$ by
\begin{equation}
\label{E.1} u(z):= \im{z}^\frac{k}{2} u_\text{h}(z).
\end{equation}
As shown in \cite[\S3.2]{MR}, $u$ is indeed a Maass cusp form
of weight $k$, trivial multiplier $v \equiv 1$ and eigenvalue
$\frac{k}{2} \left(1-\frac{k}{2}\right)$. Hence, $u$ has spectral
values $\nu \in \left\{ \frac{k-1}{2},\frac{1-k}{2}\right\}$.

The following proposition compares the period functions attached to $u$ and the period polynomial attached to $u_h$.
It is based on \cite[Proposition~49]{Mu03}.
\begin{proposition}
\label{E.2} Let $u$ be the Maass cusp form in \eqref{E.1}, which is derived
from a modular cusp form $u_\text{h}$ of weight $k \in 2\IN$.
\begin{enumerate}
\item
The function $P_{k,\frac{1-k}{2}}$ associated to $u$ by \eqref{D4.3}
vanishes everywhere.
\item
The function $P_{k,\frac{k-1}{2}}$ associated to $u$ by
\eqref{D4.3} restricted to the right half-plane $\{\zeta\in \IC;\; \re{\zeta} > 0\}$
is a multiple of the period polynomial $p$ associated to
$u_\text{h}$ by \eqref{C2.1}: $P_{k,\frac{k-1}{2}} = (2-2k) \, p$.
\end{enumerate}
\end{proposition}

\begin{proof}
Since $$\left(\frac{k-1}{2}\right)\left(\frac{1-k}{2} \right) =
\frac{k}{2}\left(1-\frac{k}{2}\right),$$ we see that $\frac{k-1}{2}$
and $\frac{1-k}{2}$ are spectral values of $u$. Moreover, $\EE^-_k
u=0$ as shown in \cite[\S3.2]{MR}.

Let $P_{k,\nu}$ be the period function associated to $u$ via \eqref{D4.3}:
\[
P_{k,\nu} (\zeta)
=
\int_0^{i\infty} \eta_{-k}(R_{-k,\nu}(\cdot,\zeta),\, u)(z).
\]
Using \eqref{D2.15} and \eqref{D2.8} we then get
\[
P_{k,\nu} (\zeta) =   \int_0^{i\infty} \left[
\left( E^+_{-k} R_{-k,\nu}(\cdot,\zeta) \right)(z) \, u(z) \frac{\dd z}{y}
 -
R_{-k,\nu}(z,\zeta)\, \left( E^-_k u  \right)(z)  \frac{\dd \bar{z}}{y}
\right].
\]
Recalling that $\EE^+_{-k} R_{-k,\nu} = (1-2\nu-k) R_{2-k,\nu}$ in
\eqref{D2.6} and $\EE^-_k u = 0$ above, we find
\begin{equation}
\label{E.3} P_{k,\nu} (\zeta) = (1-2\nu-k) \int_0^{i\infty}
R_{2-k,\nu}(z,\zeta) \, u(z) \frac{\dd z}{y}
\end{equation}
for $\re{\zeta} >0$.

To prove the first part of the Proposition, we assume $\nu
=\frac{1-k}{2}$. Then, the factor $1-2\nu-k$ in \eqref{E.3}
vanishes, implying $P_{k,\frac{1-k}{2}} =0$.

To prove the second part of the Proposition, we assume
$\nu =\frac{k-1}{2}$. By \eqref{D2.5} we have
\begin{equation}
\label{E.4}
\begin{split}
R_{2-k,\frac{k-1}{2}}(z,\zeta)
&=
\left(\frac{\sqrt{\zeta - z\,}}{\sqrt{\zeta - \bar{z}\,}}\right)^{k-2}
\left( \frac{\abs{\im{z}}}{(\zeta-z)(\zeta-\bar{z})}\right)^\frac{2-k}{2} \\
&=
(\zeta-z)^{k-2} \, \abs{\im{z}}^\frac{2-k}{2}
\end{split}
\end{equation}
for every $\zeta$ and $z$ with $\zeta-z, \zeta - \bar{z} \neq \IR_{\leq 0}$.
Combining this with \eqref{E.1} in \eqref{E.3} gives
\[
P_{k,\frac{k-1}{2}} (\zeta) =
(2-2k)  \int_0^{i\infty}  (\zeta-z)^{k-2} \, u_\text{h}(z) \dd z \\
= (2-2k) \, p(\zeta)
\]
for at least all $\zeta$ in the right half-plane.
\end{proof}

Can we also recover the periodic function $f_h$?
The answer is given in the following proposition.

\begin{proposition}
\label{E.5} Let $u$ be the Maass cusp form in \eqref{E.1}, which is derived
from a modular cusp form $u_\text{h}$ of weight $k \in 2\IN$.
\begin{enumerate}
\item
The integral transformation \eqref{D3.1} defining $f(\zeta)$ for $\zeta \in \IH$ is well-defined for both spectral
values $\nu \in \left\{ \frac{1-k}{2}, \frac{k-1}{2} \right\}$.
\item
The function $f$ associated to $u$ by \eqref{D3.1} with weight $k$ and spectral value $\nu=\frac{1-k}{2}$ vanishes
everywhere.
\item
The function $f$ associated to $u$ by \eqref{D3.1} with weight $k$ and spectral value $\nu=\frac{k-1}{2}$ and restricted
to the upper half-plane $\IH$
is a multiple of $f_\text{h}$ associated to $u_\text{h}$ by \eqref{C3.1}: $f = (2-2k) \, f_\text{h}$.
\end{enumerate}
\end{proposition}

\begin{proof}
We follow the arguments of the proof of Proposition~\ref{E.2}.
For $\zeta \in \IH$ is the nearly periodic function $f$ associated to $u$ given by \eqref{D3.1}.
Using \eqref{D2.15} and \eqref{D2.8} we then get
\[
f (\zeta) =   \int_\zeta^{i\infty} \left[
\left( E^+_{-k} R_{-k,\nu}(\cdot,\zeta) \right)(z) \, u(z) \frac{\dd z}{y}
 -
R_{-k,\nu}(z,\zeta)\, \left( E^-_k u  \right)(z)  \frac{\dd \bar{z}}{y}
\right].
\]
Recalling $\EE^+_{-k} R_{-k,\nu} = (1-2\nu-k) R_{2-k,\nu}$ in \eqref{D2.6} and $\EE^-_k u = 0$,
we find
\begin{equation}
\label{E.6} f(\zeta) = (1-2\nu-k) \int_\zeta^{i\infty}
R_{2-k,\nu}(z,\zeta) \, u(z) \frac{\dd z}{y}
\end{equation}
for $\zeta \in \IH$.

To prove the second part of the Proposition, we assume $\nu
=\frac{1-k}{2}$. Then, the factor $1-2\nu-k$ in \eqref{E.6}
vanishes, implying $P_{k,\frac{1-k}{2}} =0$.

To prove the third part of the Proposition, we assume
$\nu =\frac{k-1}{2}$.
Using \eqref{E.4} and \eqref{E.1} in \eqref{E.6} gives
\[
f (\zeta) =
(2-2k)  \int_\zeta^{i\infty}  (\zeta-z)^{k-2} \, u_\text{h}(z) \dd z \\
= (2-2k) \, f_\text{h}(\zeta)
\]
for every $\zeta \in \IH$.

The first part follows also from the above calculations. Even if the
calculations above are a priori formal, the well-definiteness of the
results show that the original integral transforms are also well
defined. We are just adding cleverly zeros.
\end{proof}

\section{Discussion and Outlook}
\label{F}
In this paper, we introduced and discussed Eichler integrals attached to Maass cusp forms of half-integral weight.
We also introduced the corresponding period functions.
This generalizes on one hand the classical case of period polynomials and periodic functions associated to
holomorphic modular cusp forms, as shown in \S\ref{E}.
On the other hand, our results fit neatly with the also known case of Maass cusp forms of weight $0$ and associated
periodic and period functions, discussed in \cite{LZ01}.

Obvious remaining questions are
\begin{enumerate}
\item
For half-integral weight, do the period functions (i.e., the space of holomorphic solutions of the three-term
equation~\eqref{D3.16} which satisfy the growth condition~\eqref{D4.9}) bijectively
correspond to Maass cusp forms?
We only show one direction.
\item
Can we use the introduced period functions to describe a ``Eichler-Shimura-cohomology''
for the half-integral or real weight case?
\item
Does everything also hold for real or complex weights and/or non-cuspidal forms?
For example, can we extend the results to the general Maass wave forms introduced in \cite{MR}?
\item
What can we say about Eichler-Shimura theory of harmonic Maass wave forms?
\item
How are period functions and $L$-series related.
\end{enumerate}
The first question is positively answered for Maass cusp forms of
weight $0$ in \cite{LZ01} and for real weights in \cite{Mu03}.
Also, Bruggeman, Lewis and Zagier discuss recently the case of Maass forms
(of weight 0 and of polynomial growth in the cusps) and is associated group cohomology in \cite{BLZ}.
Deitmar and Hilgert discuss the situation for subgroups of finite index and weight 0 in $\SL{\IZ}$ in \cite{DH07}.
A recent result by Deitmar discusses the situation for Maass wave forms of higher order in \cite{De11}.

To our knowledge, the second and third question are still open for general Maass wave forms with complex weight.
Our results extend trivially to the case of Maass cusp forms with real weight (by just replacing half-integer with real everywhere).
The third question is also positively answered in \cite{KR10} for generalized modular forms (introduced in \cite{KM03}).

The fourth question is answered in \cite{BGKO}.
They show an Eichler-Shimura-type result for harmonic Maass wave forms, see e.g.\ \cite[Theorem~1.2]{BGKO}.

The last question is also discussed in \cite{Mu03}, generalizing the first part of \cite{LZ01}.

\section*{Acknowledgements}
The authors would like to thank the referee for the excellent
recommendations. In addition, the authors would like to thank the
Center for Advanced Mathematical Sciences (CAMS) at the American
University of Beirut for the support. The first-named author would
like to express his gratitude towards CAMS for supporting his visit.


\bibliographystyle{amsalpha}

\end{document}